\documentclass[12pt]{amsart}

\usepackage{amsmath}
\usepackage{amsfonts}
\usepackage{amssymb}
\usepackage{graphicx}
\usepackage{mathrsfs}
\usepackage{pb-diagram}
\usepackage{epstopdf}
\usepackage{amsmath,amsfonts,amsthm,enumerate,amscd,latexsym,curves}
\usepackage{bbm}
\usepackage[mathscr]{eucal}
\usepackage{epsfig,epsf,epic}
\usepackage{caption}
\usepackage{subcaption}

\newtheorem{thm}{Theorem}[section]
\newtheorem{lemma}[thm]{Lemma}
\newtheorem{cor}[thm]{Corollary}
\newtheorem{prop}[thm]{Proposition}
\newtheorem{defn}[thm]{Definition}

\newtheorem{remark}[thm]{Remark}

\numberwithin{equation}{section}

\newcommand{\conn}{\nabla}

\newcommand{\der}{\mathrm{d}}

\newcommand{\N}{\mathbb{N}}
\newcommand{\Z}{\mathbb{Z}}

\newcommand{\R}{\mathbb{R}}
\newcommand{\C}{\mathbb{C}}
\newcommand{\cpx}{\mathbb{C}}
\newcommand{\bP}{\mathbb{P}}

\newcommand{\bu}{\boldsymbol{u}}
\newcommand{\bv}{\boldsymbol{v}}

\newcommand{\bX}{\mathcal{X}}

\newcommand{\unu}{\underline{\nu}}
\newcommand{\bS}{\mathbb{S}}
\newcommand{\bF}{\mathbb{F}}

\newcommand{\CM}{\mathcal{M}}
\newcommand{\CO}{\mathcal{O}}

\newcommand{\pt}{\mathrm{pt}}
\newcommand{\Cone}{\mathrm{Cone}}
\newcommand{\ev}{\mathrm{ev}}

\begin{document}

\title{Open Gromov-Witten invariants and SYZ under local conifold transitions}
\author[Lau]{Siu-Cheong Lau}
\address{Department of Mathematics\\ Harvard University\\ One Oxford Street\\ Cambridge \\ MA 02138\\ USA}
\email{s.lau@math.harvard.edu}

\begin{abstract}
For a local Calabi-Yau manifold which arises as smoothing of a toric Gorenstein singularity, this paper derives the open Gromov-Witten invariants of a generic fiber of the special Lagrangian fibration constructed by Gross and thereby constructs its SYZ mirror.  Moreover it shows that the SYZ mirror of a local Calabi-Yau manifold and that of its conifold transition can be analytic continued to each other, which gives a global picture of SYZ mirror symmetry in this setup.
\end{abstract}

\maketitle

\section{Introduction}

Let $N$ be a lattice and $P$ a lattice polytope in $N_\R$.  It is an interesting classical problem in combinatorial geometry to construct Minkowski decompositions of $P$.

On the other hand, consider the family of polynomials $\sum_{v \in P \cap N} c_v z^v$ for $c_v \in \C$ with $c_v \not= 0$ when $v$ is a vertex of $P$.  All these polynomials have $P$ as their Newton polytopes.   This establishes a relation between geometry and algebra, and the geometric problem of finding Minkowski decompositions is transformed to the algebraic problem of finding polynomial factorizations.

One purpose of this paper is to show that the beautiful algebro-geometric correspondence between Minkowski decompositions and polynomial factorizations can be realized via \emph{SYZ mirror symmetry}.

A key step leading to the miracle is brought by a result of Altmann \cite{altmann}.  First of all, a lattice polytope corresponds to a toric Gorenstein singularity $X$.  Altmann showed that Minkowski decompositions of the lattice polytope correspond to (partial) smoothings of $X$.  From string-theoretic point of view different smoothings of $X$ belong to different sectors of the same stringy K\"ahler moduli.  Under mirror symmetry, this should correspond to a complex family of local Calabi-Yau manifolds, and the various sectors (coming from Minkowski decompositions) correspond to various (conifold) limits of the complex family.

We show that such complex family can be realized via SYZ construction by using special Lagrangian fibrations on smoothings of toric Gorenstein singularities constructed by Gross \cite{gross_examples}.  This paper follows the framework of SYZ given in \cite{CLL} which uses open Gromov-Witten invariants rather than their tropical analogs.  \emph{All the relevant open Gromov-Witten invariants are computed explicitly} (Theorem \ref{openGW_Fr}), which gives a more direct understanding to symplectic geometry.

From the computations we obtain an explicit expression of the SYZ mirror:

\begin{thm}[see Theorem \ref{SYZ thm} for the detailed statement] \label{SYZthm1}
For a smoothing of a toric Gorenstein singularity coming from a Minkowski decomposition of the corresponding polytope, its SYZ mirror is
$$ uv = \prod_{i=0}^p \left(1 + \sum_{l = 1,\ldots,k_i} z^{u^i_l} \right) $$
where $u^i_l$ are vertices of simplices appearing in the Minkowski decomposition.
\end{thm}

See Section \ref{Gfib} and \ref{SYZ} for more details on Minkowski decomposition and related notations.  Thus the SYZ mirror gives a factorization of a polynomial corresponding to a Minkowski decomposition of $P$.  The following diagram summarizes the dualities that we have and their relations:

\begin{figure}[htp]
\begin{center}
\includegraphics[scale=0.75]{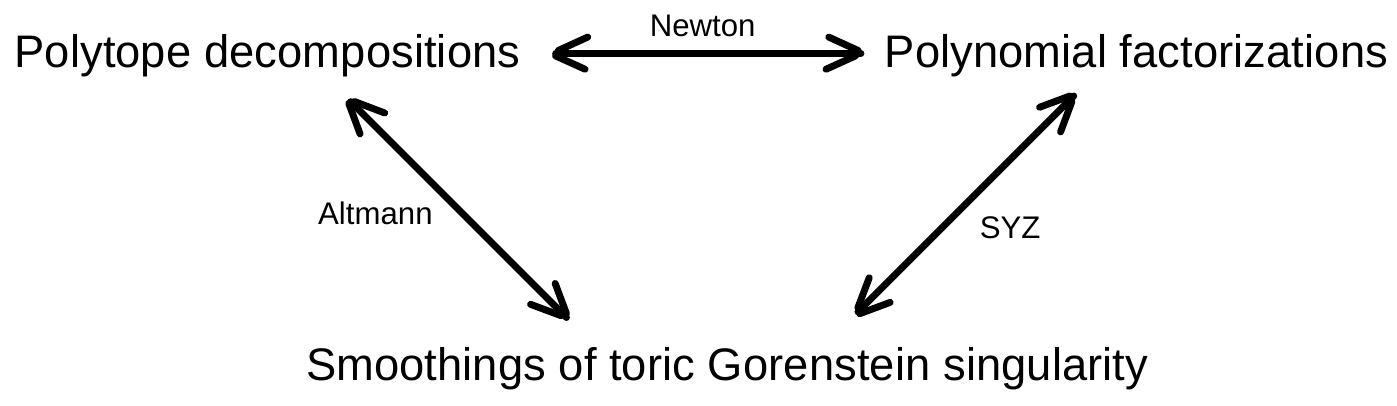}
\end{center}
\end{figure}

Another motivation of this paper comes from \emph{global SYZ mirror symmetry}.  Currently most studies in SYZ mirror symmetry focus on the large complex structure limit, and other limit points of the moduli are less understood.  Open Gromov-Witten invariants and SYZ for toric Calabi-Yau manifolds were studied in \cite{CLL,CLT11,CCLT13}, which are around the large complex structure limit.  This paper studies SYZ for conifold transitions of toric Calabi-Yau manifolds, which are other limit points of the moduli.  This gives a more global understanding of SYZ mirror symmetry.

More concretely, we prove the following:

\begin{thm}[see Theorem \ref{thm:ct} for the complete statement]
Let $Y$ be a toric Calabi-Yau manifold associated to a triangulation of a lattice polytope $P$, and $\bX_t$ be the conifold transition of $Y$ induced by a Minkowski decomposition of $P$.

Then their SYZ mirrors $\check{\bX}$ and $\check{Y}_q$ are connected by an analytic continuation: there exists an invertible change of coordinates $q(\check{q})$ and a specialization of parameters $\check{q} = \underline{\check{q}}$ such that
$$\check{\bX} = \check{Y}_{q(\check{q})}|_{\check{q} = \underline{\check{q}}}.$$
\end{thm}

The above theorem is proved by combining Theorem \ref{SYZthm1} and the open mirror theorem for a toric Calabi-Yau manifold given by \cite{CCLT13}, which gives an explicit expression of the disc potential and SYZ mirror in terms of the Hori-Vafa mirror and the mirror map.  This is similar to the strategy of proving Ruan's crepant resolution conjecture via mirror symmetry \cite{iritani09,CIT}.

Notice that \emph{the smoothings $\bX_t$ are no longer toric}.  Open Gromov-Witten invariants for non-toric cases are usually difficult to compute and only known for certain isolated cases (such as the real Lagrangian in the quintic \cite{PSW}).  This paper gives a class of non-toric manifolds that the open Gromov-Witten invariants can be explicitly computed.

In this paper we do not need Kuranishi's obstruction theory to define the invariants: there is no non-constant holomorphic sphere (Lemma \ref{no sphere}), and hence no sphere bubbling can occur.  (And no disc-bubbing occurs since we consider discs with the minimal Maslov index.) This makes the theory of open Gromov-Witten invariants much easier.  On the other hand, the local Calabi-Yau here is non-toric, and so the computation involves non-trivial arguments.  

One can also construct the mirror via Gross-Siebert program \cite{gross07} using tropical geometry of the base of the fibration, which will give the same answer.  In this case Gross-Siebert program is rather easy to run since all the walls are parallel and so they do not interact with each other.  Nevertheless, since the correspondence between tropical and symplectic geometry in the open sector is still conjectural, this paper takes the approach based on symplectic geometry instead and requires a non-trivial computation of open Gromov-Witten invariants.

There are several interesting related works to the author's knowledge.  The work by Castano-Bernard and Matessi \cite{CM} gave a detailed treatment of Lagrangian fibrations and affine geometry of the base under conifold transitions, and studied mirror symmetry in terms of tropical geometry along the line of Gross-Siebert.  For $A_n$-type surface singularities and the three-dimensional local conifold singularity $\{xy = zw\}$, SYZ construction by tropical wall-crossing was studied by Chan-Pomerleano-Ueda \cite{Chan-An,Chan-Ueda,CPU}.  Auroux \cite{auroux07,auroux09} studied wall-crossing of open Gromov-Witten invariants and the SYZ mirror of the complement of an anti-canonical divisor and gave several beautiful and illustrative examples. Abouzaid-Auroux-Katzarkov \cite{AAK} gave a beautiful treatment of SYZ for blowups of toric varieties which is useful for studying mirror symmetry for hypersurfaces in toric varieties.  Relation between open Gromov-Witten invariants of the Hirzebruch surface $\bF_2$ and its conifold transition was studied by Fukaya-Oh-Ohta-Ono \cite{FOOO10} with an emphasis on non-displaceable Lagrangian tori (which are not fibers).

A related story in the Fano setting was studied by Akhtar-Coates-Galkin-Kasprzyk \cite{ACGK}, where they considered Minkowski polynomials and mutations in relation with mirror symmetry for Fano manifolds.  This paper works with local Calabi-Yau manifolds instead, and stresses more on open Gromov-Witten invariants and SYZ constructions.


\section*{Acknowledgment}
I am grateful to N.C. Conan Leung for encouragement and useful advice during preparation of this paper.  I also express my gratitude to my collaborators for their continuous support, including K. Chan, C.-H. Cho, H. Hong, H.H. Tseng and B.-S. Wu.  The work contained in this paper was supported by Harvard University.

\section{Smoothing of toric Gorenstein singularity by Minkowski decomposition} \label{smoothing}
Let $N$ be a lattice, $M$ be the dual lattice, and $\unu \in M$ be a primitive vector.  Let $P$ be a lattice polytope in the affine hyperplane $\{v \in N_\R: \unu(v) = 1\}$.  Then $\sigma = \Cone(P) \subset N_\R$ is a Gorenstein cone.  We denote by $m$ the number of corners of the polytope $P$, and by $\tilde{m}$ the number of lattice points contained in the (closed) polytope $P$.  The corresponding toric variety $X = X_\sigma$ using $\sigma$ as the fan is a toric Gorenstein singularity.  We assume that the singularity is an isolated point.  

We choose a lattice point $v_0$ in the (closed) lattice polytope $P$ to translate it to a polytope in the hyperplane $\unu^\perp_\R \subset N_\R$, and by abuse of notation we still denote this by $P$.

By Altmann \cite{altmann}, from a Minkowski decomposition
$$P = R_0 + R_1 + \ldots + R_p$$
where $R_i$'s are convex subsets in $\unu^\perp_\R$ and $p \in \N$,  one obtains a (partial) smoothing of $X$ as follows.  Let $\hat{N} := (\unu^\perp) \oplus \Z^{p+1}$.  Define $\hat{\sigma}$ to be the cone
$$ \hat{\sigma} = \Cone \left( \bigcup_{i=0}^p \left(R_i \times \{e_i\}\right) \right) $$
where $\{e_i\}_{i=0}^p$ is the standard basis of $\Z^{p+1}$.  The total space of the family is $\bX = X_{\hat{\sigma}}$, and $\bX_t$ is defined as fibers of
$$[t_0-t_1, \ldots, t_0-t_p]:\bX = X_{\hat{\sigma}} \to \cpx^{p+1} / \cpx \langle (1,\ldots,1) \rangle$$
with $\bX_0 = X$.  Here $t_0, \ldots, t_p$ are functions corresponding to $(0,\check{e}_i) \in \hat{M}$, where $\hat{M}$ is the dual lattice to $\hat{N}$ and $\{\check{e}_i\}_{i=0}^p$ is dual to the standard basis $\{e_i\}_{i=0}^p$ of $\Z^{p+1}$.

The total space $\bX = X_{\hat{\sigma}}$ of the family is a toric variety, but a generic member $\bX_t$ is \emph{not} toric.  Each member $\bX_t$ is equipped with a K\"ahler structure $\omega_t$ induced from $\bX$ and a holomorphic volume form $\underline{\Omega}_t := (\iota_{t_1} \ldots \iota_{t_p}) \underline{\hat{\Omega}} |_{\bX_t}$, where $$\underline{\hat{\Omega}} = \der \log z_0 \wedge \ldots \wedge \der\log z_{n-1} \wedge \der t_0 \wedge \ldots \wedge \der t_p$$
is the volume form on $\bX$.

From now on we assume that $\bX_t$ is smooth for generic $t$.  In particular we assume that every summand $R_i$ is a unimodular $k_i$-simplex for some $k_i = 1, \ldots, n-1$.\footnote{Here a $k$-simplex with one of its vertex being $0$ is said to be unimodular if its vertices generates all the lattice points in the $k$-plane containing the simplex.}  The vertices of $R_i$ are $0, u^i_1,\ldots,u^i_{k_i}$ for some lattice points $u^i_j \in \unu^\perp$.

\section{Gross fibration and its variants} \label{Gfib}
Let us fix the smoothing parameter $c = [c_0, c_1, \ldots, c_p] \in \cpx^{p+1} / \cpx \langle (1,\ldots,1) \rangle$ from now on, where $c_0 = 0$, $c_1, \ldots, c_p \in \cpx$ are taken to be distinct constants, and consider $\bX_{c}$ which is assumed to be smooth.  By relabelling $c_i$'s if necessary, we assume $0 = |c_0| \leq |c_1| \leq \ldots \leq |c_p|$.  The functions $t_i$'s restricted to $\bX_{c}$ equals to $t + c_i$, where $t := t_0$.

Gross \cite{gross_examples} constructed a special Lagrangian fibration on $\bX_{c}$, which is
$$ \pi^K := (\pi_0, |t - K|^2 - K^2): \bX_c \to B$$
where
$$ B = \frac{M_\R}{\R\langle \unu \rangle} \times \R_{\geq -K^2},$$
$K \in \R_{>0}$, $\pi_0$ is the moment map of the action $T^{\perp \unu}$ on $(\bX_c, \omega_c, \Omega_c)$.  Here $\omega_c$ is the symplectic form restricted from $\bX$ to $\bX_c$, and
\begin{equation} \label{Omega}
\Omega_{c} = \iota_{t_1} \ldots \iota_{t_p} \hat{\Omega}|_{\bX_{c}} = \der \log z_0 \wedge \ldots \wedge \der \log z_{n-1} \wedge \der \log (t-K)
\end{equation}
and
\begin{equation}\hat{\Omega} = \der \log z_0 \wedge \ldots \wedge \der \log z_{n-1} \wedge \der \log (t_0-K) \wedge \der t_1 \wedge \ldots \wedge \der t_p
\end{equation}
on $\bX$, which are \emph{different} from the ones we described in the end of Section \ref{smoothing}.  $\Omega_c$ is nowhere zero and has a simple pole of order one on the boundary divisor $D_0 = \{t = K\} \subset \bX_c$.  As $K$ tends to $\infty$, $\pi^K$ tends to the fibration
$$ \pi^\infty := (\pi_0, \mathrm{Re} \,t): \bX_c \to B$$
which is no longer proper, and we will not use $\pi^\infty$ in this paper.\footnote{On the side of toric resolutions, the fibration $\pi^\infty$ and the related open Gromov-Witten invariants in dimension three has been well-studied by the theory of topological vertex \cite{LLMZ}.}

By Proposition 3.3 of \cite{gross_examples}, away from the boundary $\partial B$, the discriminant loci have codimension two, and they lie in the hyperplanes
$$H_i =  \frac{M_\R}{\R\langle \unu \rangle} \times \{ |c_i - K|^2 - K^2 \}.$$
The level sets of $|t - K|^2 - K^2$ gives a fibration of the complex $t$-plane over $\R_{\geq 0}$ by circles.  For a singular fiber of $\pi^K$, its image under $t$ is a circle centered at $K$ hitting one of the $c_i$'s.  See Figure \ref{concentric}.

\begin{figure}[htp]
\begin{center}
\includegraphics[scale=0.6]{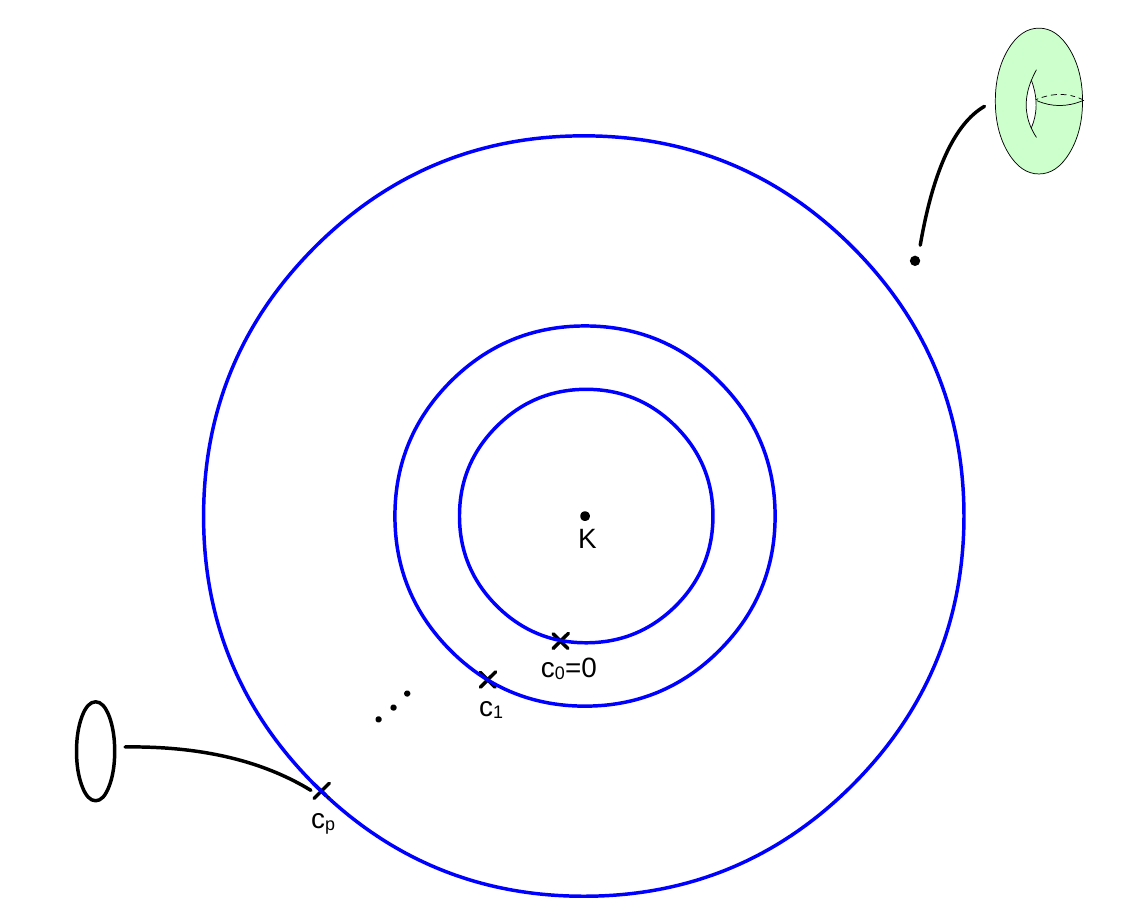}
\end{center}
\caption{The $t$ plane.  The image of a fiber of $\pi^K$ is a circle centered at $K$, and that of a singular fiber hits one of the $c_i$'s.  A torus $T^{\perp \unu}$ is attached to each point other than $c_i$'s.  At $c_i$'s the $T^{\perp \unu}$-orbit may degenerate to a smaller torus.}
\label{concentric}
\end{figure}

The boundary divisor $\pi^{-1} (\partial B) = \{t = K\}$ is denoted as $D_0$, and the discriminant loci supported in $H_i$ are denoted as $\Gamma_i$.  For a generic choice of $K$, $|c_i - K|$'s are pairwise distinct for $i=0,\ldots,p$ and so all the hyperplanes $H_i$'s and $\partial B$ are disjoint.  In the rest of this paper we assume such a choice of $K$.

We can deform the Lagrangian fibration $\pi^K$ and make all walls $H_i$ collide into one: deform the metric to $d(\cdot,\cdot)$ on the complex $t$ plane from the standard norm $|\cdot|$, so that all $c_i$'s are of the same distance to $K$ for $i=0,\ldots,p$.  Then the Lagrangian fibration $d(\cdot, K)^2 - K^2$ on the $t$-plane pulls back to a Lagrangian fibration $\underline{\pi}$ on $\bX_t$.  Since all $c_i$'s lie in the same fiber, the discriminant locus away from the boundary $\partial B$ lies in a single hyperplane $H$.  The fibers are no longer special with respect to $\Omega_c$, but this does not hurt since we are only concerned with symplectic geometry.  Interesting singular Lagrangian fibers arise for the fibration $\underline{\pi}$.  We will study the disc potential of a product torus fiber $T$ (see Definition \ref{defn:W^T} and Corollary \ref{W^T}) of $\underline{\pi}$ and its behavior under conifold transitions (Theorem \ref{thm:ct}).

For the purpose of SYZ construction along the line of \cite{CLL}, we consider the following compactification of $\bX_c$.  Add the ray generated by $-e_0 - \ldots - e_p$ to the fan $\hat{\sigma}$ and corresponding cones to produce a complete fan $\bar{\sigma}$.  Then consider the corresponding toric variety $\bar{\bX} =  X_{\bar{\sigma}}$ and fibers $\bar{\bX}_c \subset \bar{\bX}$ of $(t_1-t_0, \ldots, t_p-t_{0})$ as before, which are preserved under the action of $T^{\perp \unu}$.  The Gross fibration has the same definition $\bar{\pi}^K = (\pi_0,|t-K|^2 - K^2)$ on $\bar{\bX}_c$, which has boundary divisor $D_0 = \{t=K\}$ and $D_\infty = \{t=\infty\}$ (the pole divisor of $t$).  The meromorphic volume form $\Omega_c$ defined by the same expression \eqref{Omega} has simple poles along $D_0$ and $D_\infty$.

\section{SYZ of smoothings of toric Gorenstein singularities} \label{SYZ}

Here we follow the construction of \cite{CLL} to obtain the SYZ mirror from disc countings.  Adapted to this case, the procedures are:
\begin{enumerate}
\item Define the semi-flat mirror to be
$$\check{X}_0 = \{\textrm{flat } U(1) \textrm{ connections } \conn \textrm{ on a fiber of } \bar{\pi}^K \textrm{ over } r \in B_0 \} $$
where $B_0 = B - \Gamma$, $\Gamma$ is the discriminant locus.  We have the bundle map $\check{\pi}:\check{X}_0 \to B_0$.  We have semi-flat complex coordinates on $\check{X}_0$.  Each $v \in \unu^\perp$ corresponds to a coordinate $z^v$ on $\check{X}_0$ which is monodromy-invariant.
\item Let $H\subset B_0$ be the wall (see Definition \ref{defn:wall}).  Define the generating functions of open Gromov-Witten invariants (see Definition \ref{defn:oGW})
\begin{equation} \label{u}
\bu = \sum_{\beta \cdot D_0 = 1} n_\beta \exp \left(-\int_\beta \omega\right) \mathrm{Hol}_\conn (\partial \beta)
\end{equation}
and
\begin{equation} \label{v}
\bv = \sum_{\beta \cdot D_\infty = 1} n_\beta \exp \left(-\int_\beta \omega\right) \mathrm{Hol}_\conn (\partial \beta)
\end{equation}
on $\check{\pi}^{-1}(B_0-H)$ associated to $D_0$ and $D_\infty$.  They serve as quantum-corrected complex coordinates.
\item Take $R$ to be the ring generated by $\bu,\bv$ and $z^v$ for $v \in \unu^\perp$.  Then the SYZ mirror is defined to be $\mathrm{Spec}(R)$.
\end{enumerate}
The readers are referred to Section 2 of \cite{CLL} for more details.

Since $\pi^K$ is a special Lagrangian fibration with respect to $\Omega_c$ which has simple poles along $D_0$ (and $\bar{\pi}^K$ is special Lagrangian with respect to $\Omega_c$ which has simple poles along $D_0$ and $D_\infty$), by \cite{auroux07} we have the following formula for the Maslov index of a disc class:

\begin{lemma}[Lemma 3.1 of \cite{auroux07}] \label{mu}
The Maslov index of $\beta \in \pi_2(\bX_c,F_r)$, where $F_r$ is a fiber of $\pi^K$ over $r \in B$, is
$ \mu (\beta) = 2 \beta \cdot D_0. $
Similarly the Maslov index of $\beta \in \pi_2(\bar{\bX}_c,F_r)$, where $F_r$ is a fiber of $\bar{\pi}^K$ is
$ \mu (\beta) = 2 \beta \cdot (D_0+D_\infty). $
\end{lemma}

In the following lemma, we see that there is no holomorphic sphere in $\bX_c$ (or $\bar{\bX}_c$).  This simplifies a lot the moduli theory for open Gromov-Witten invariants because no sphere bubbling occurs.

\begin{lemma} \label{no sphere}
There is no non-constant holomorphic map from $\bP^1$ to $\bX_c$ (or $\bar{\bX}_c$). 
\end{lemma}

\begin{proof}
Suppose $u: \bP^1 \to \bX_c$ (or $u: \bP^1 \to \bar{\bX}_c$) is a holomorphic map.  Then $t \circ u$ is a holomorphic function on $\bP^1$ which can only be constant.  Then the image of $u$ lies in the level set of $t$.  But $H_2(t^{-1}\{a\}) = 0$ for all $a \in \cpx$, and hence $u$ can only be a constant.
\end{proof}

For each $i$, we may identify $H_i$ with $M_\R / \R\langle \unu \rangle$.  Then the discriminant locus $\Gamma_i \subset H_i$ is the normal fan of the simplex $R_i$.  Thus $H_i - \Gamma_i$ consists of $k_i+1$ components, which are one-to-one corresponding to the vertices of the standard simplex $R_i$.  Each chamber is (up to translation) the dual cone of $(R_i)_{v}$, where $v$ is the corresponding vertex of $R_i$.

From now on we assume that $|c_{i+1} - K| \geq |c_{i} - K|$ for all $i = 0, \ldots, p-1$, and $K>0$ is chosen generically such that all the hyperplanes $H_i$ for $i = 1,\ldots,p$ are distinct.

To obtain a trivialization of the torus bundle $(\pi^K)^{-1}(B_0) \to B_0$, we fix a chamber in $H_i$ for each $i$.  Without loss of generality for each $i$, we fix the chamber corresponding to the vertex $0$ of the simplex $R_i$.  This gives a simply-connected open set $U \subset B_0$ and $\bX_t$ over $U$ is a trivial torus bundle.  The normal vectors to the facets of the chamber are $u^i_1,\ldots,u^i_{k_i} \in \unu^\perp$.  We always use this trivialization in the rest of the paper.

Now we consider open Gromov-Witten invariants bounded by a fiber of $\pi^K$ above points in $U \subset B$.  The invariants are well-defined when the fiber has minimal Maslov index two.  Recall the following definition of wall from \cite{CLL}:

\begin{defn}[Wall] \label{defn:wall}
The wall $H$ of the Lagrangian fibration $\pi^K: \bX_t \to B$ (or $\bar{\pi}^K$) is the set
$$ \{ r \in B_0: F_r \textrm{ bounds non-constant holomorphic discs with Maslov index zero} \}.$$
\end{defn}

Away from the wall, the one-pointed genus-zero open Gromov-Witten invariants are well-defined, and we can use them to cook up the SYZ mirror.

\begin{defn}[Open Gromov-Witten invariants] \label{defn:oGW}
Let $r \in B_0 - H$ be away from wall $H$, and $\beta \in \pi_2(\bar{\bX}_t,F_r)$ a disc class bounded by the fiber $F_r$.  We have the moduli space $\CM_1(\beta)$ of stable discs with one boundary marked point representing $\beta$.  The open Gromov-Witten invariant associated to $\beta$ is
$$ n_\beta = \int_{\CM_1(\beta)} \ev^*[\pt] $$
where $\ev:\CM_1(\beta) \to F_r$ is the evaluation map at the boundary marked point.
\end{defn}

The open Gromov-Witten invariant $n_\beta$ is non-zero only when the Maslov index $\mu(\beta)$ is two.  Moreover, by Lemma \ref{no sphere}, away from the wall there is no disc bubbling and sphere bubbling in the moduli.  Thus $\CM_1(\beta)$ consists of (classes of ) holomorphic maps $(\Delta,\partial\Delta) \to (X,F_r)$.  No virtual perturbation theory is needed in this situation.

In this case the wall is a union of disjoint hyperplanes as stated in the following proposition.

\begin{prop} \label{prop:wall}
The wall of the Lagrangian fibration $\pi^K: \bX_t \to B$ (or $\bar{\pi}^K$) is the union of the hyperplanes $H_i$ for $i = 0,\ldots,p$, which are disjoint from each other.
\end{prop}

\begin{proof}
Suppose $F_r$ bounds a non-constant holomorphic disc $u: (\Delta, \partial \Delta) \to (\bX_c, L)$ with Maslov index zero.  By the Maslov index formula (Lemma \ref{mu}), the disc does not hit the divisor $D_0$.  This means the holomorphic function $(t - K) \circ u$ is never zero.  Since $|t - K| \circ u$ is constant on $\partial \Delta$, by the maximal principle it follows that $t$ is a constant.  Thus the disc $u$ lies in the level set of $t$, and by topology this forces $t = 0, c_1,\ldots,c_p$.  This implies $r$ lies in one of the hyperplanes $H_i$'s.
\end{proof}

We need to identify all the disc classes in order to compute their open Gromov-Witten invariants.  The next definition gives a label to every basic disc class.

\begin{defn}[Disc classes]
Let $F_r$ be a fiber of the Lagrangian fibration $\bar{\pi}^K$ contained in the trivialization $(\bar{\pi}^K)^{-1}(U)$.  The disc class of Maslov index two emanated from the boundary divisor $D_0$ (or $D_\infty$) is denoted as $\beta_0$ (or $\beta_\infty$ respectively).  Moreover, each discriminant locus $\Gamma_i \subset H_i$ gives rise to $k_i$ disc classes of Maslov index zero, which are in one-to-one correspondence with the normal vectors $u^i_1,\ldots,u^i_{k_i} \in \unu^\perp$ to the facets of the chamber, and they are denoted as $\beta^i_j$ for $j = 1, \ldots, k_i$.

The disc class $\beta_0$ gives a local cooordinate function $z_0$ on the semi-flat mirror $\check{\pi}^{-1}(U) \subset \check{X}_0$ (which is not monodromy invariant):
$$ z_0 := \exp \left( -\int_\beta \omega \right) \mathrm{Hol}_{\conn} (\partial\beta).$$
\end{defn}

All stable disc classes bounded by $F_r$ of Maslov index two must be of the form
$$ \beta_0 + \sum_{i=0}^p \sum_{j=1}^{k_i} n^i_j \beta^i_j \,\,\textrm{ or }\,\, \beta_\infty + \sum_{i=0}^p \sum_{j=1}^{k_i} n^i_j \beta^i_j $$ 
where $n^i_j \in \Z_{\geq 0}$.  In the following theorem we classify all the stable disc classes of Maslov index two and compute their open Gromov-Witten invariants.  In the toric case, an analogous result is obtained by \cite{cho06}.  Since here the symplectic manifold $\bX_t$ (or $\bar{\bX}_t$) is non-toric and the Lagrangian-fibration $\pi^K$ (or $\bar{\pi}^K$) has interior discriminant loci, we need a non-trivial argument here.

\begin{thm} \label{openGW_Fr}
Let $r$ be a regular value between the walls $H_l$ and $H_{l+1}$ of the Lagrangian fibration $\bar{\pi}^K$, and let
$$ \beta_0 + \sum_{i=0}^p \sum_{j=1}^{k_i} n^i_j \beta^i_j \,\,\textrm{ or }\,\, \beta_\infty + \sum_{i=0}^p \sum_{j=1}^{k_i} n^i_j \beta^i_j $$
be a disc class which has Maslov index two bounded by $F_r$.  Then the moduli space $\CM_1(\beta)$ is non-empty only when $n^i_j = 0$ for all $i,j$ except for each $0 \leq i \leq l$, there could be at most one $j = 1,\ldots, k_i$ with $n^i_j = 1$.  In such a case the open Gromov-Witten invariant is $n_{\beta} = 1$.
\end{thm}

\begin{proof}
First consider the case $\beta_0 + \sum_{i=0}^p \sum_{j=1}^{k_i} n^i_j \beta^i_j$.

It follows from the maximal principle that the moduli space $\CM_1(\beta)$ is non-empty only when $n^i_j = 0$ for all $i > l$.

Since there is no holomorphic sphere in $\bX_c$ (Lemma \ref{no sphere}), no sphere bubbling occurs in the moduli.  Also there is no disc bubbling because the fiber under consideration is not lying over the walls.  Thus elements in $\CM_1(\beta)$ are holomorphic maps $u:(\Delta, \partial \Delta) \to (\bX_t,F_r)$.  Using local coordinates $(\zeta_1,\ldots,\zeta_{n-1},t)$ of $\bX_c$ (where $\zeta_1,\ldots,\zeta_{n-1}$ corresponds to a basis of $M / \Z\langle \unu \rangle$), the map $u$ can be written as $(\zeta_1(z),\ldots,\zeta_{n-1}(z),t(z))$ for $z \in \Delta$.  Moreover $F_r$ is defined by the equations
$$|\zeta_1|^2 = r_1, \ldots, |\zeta_{n-1}|^2 = r_{n-1}, |t-K|^2 = K^2 + r_n.$$

Consider the holomorphic function $t \circ u$ on the disc $\Delta$.  Since $u$ represents $\beta_0 + \sum_{i=0}^p \sum_{j=1}^{k_i} n^i_j \beta^i_j$, the winding number of $t \circ u$ around $K$ is $1$.  Since $r$ is lying between the walls $H_l$ and $H_{l+1}$, the winding numbers of $t \circ u$ around $c_0 = 0$ and $c_1, \ldots, c_l$ are also $1$.  Thus $t \circ u$ attains each value $c_0, \ldots, c_l$ exactly once.  On the other hand, $u$ hits the critical fibers over $\Gamma_i$ for $\sum_{j=1}^{k_i} n^i_j$ times (counting with multiplicity), and so $t \circ u$ attains the value $c_i$ for at least $\sum_{j=1}^{k_i} n^i_j$ times.  This forces $n^i_j$ is either zero or one, and among $n^i_1,\ldots, n^i_{k_i}$'s at most one of them is one for each $i$.


Suppose this is the case.  Then we equate $\CM_1(\beta)$ and $\CM_1(\beta_0)$ as follows.  Let $u:(\Delta, \partial \Delta) \to (\bX_t,F_r)$ represent the class $\beta$.  For $i=0$, extend the vertices $u^0_1,\ldots, u^0_{k_0}$ of the polytope $R_0$ to a basis of $\unu^\perp$ and let $\zeta_1,\ldots,\zeta_{n-1}$ be functions corresponding to the dual basis of $M / \Z\langle \unu \rangle$.  $u$ can be written as $(\zeta_1(z),\ldots,\zeta_{n-1}(z),t(z))$ in terms of these coordinates.  If $n^0_j = 1$ for a certain $j = 1, \ldots, k_0$ (otherwise we do nothing), then there exists $z_0 \in \Delta^{\circ}$ such that $\zeta_{j}(z_0) = 0$ and $t(z_0) = 0$.  Define
$$ \tilde{u}(z) = \left(\zeta_1(z),\ldots, \left(\frac{z-z_0}{1+\bar{z_0} z}\right)^{-1} \zeta_{j}(z) ,\ldots,\zeta_{n-1}(z),t(z)\right). $$
Then $\tilde{u}$ does not hit any singular fibers at $\Gamma_0$ and hence belongs to the moduli space $\CM_1(\beta_0 + \sum_{i=1}^p \sum_{j=1}^{k_i} n^i_j \beta^i_j)$.  Conversely let $\tilde{u}:(\Delta, \partial \Delta) \to (\bX_t,F_r)$ represent the class $\beta_0 + \sum_{i=1}^p \sum_{j=1}^{k_i} n^i_j \beta^i_j$.  By the above analysis on winding number of $t$, there exists exactly one $z_0 \in \Delta^{\circ}$ such that $t(z_0) = 0$.  Then define
$$ u(z) = \left(\zeta_1(z),\ldots, \left(\frac{z-z_0}{1+\bar{z_0} z}\right) \zeta_{j}(z) ,\ldots,\zeta_{n-1}(z),t(z)\right) $$
which belongs to $\CM_1(\beta)$.  This set up an isomorphism $$\CM_1(\beta) \cong \CM_1\left(\beta_0 + \sum_{i=1}^p \sum_{j=1}^{k_i} n^i_j \beta^i_j\right).$$
Do the same thing for $i=1,\ldots, p$, inductively we obtain an isomorphism between $\CM_1(\beta)$ and $\CM_1(\beta_0)$.

Now for the class $\beta_0$, by the maximum principle any holomorphic map in $\beta_0$ have coordinates $z_1, \ldots, z_{n-1}$ being constants.  The moduli space $\CM_1(\beta_0)$ does not have disc bubbling even for a Lagrangian fiber at the walls $H_i$'s, and thus it remains to be the same under wall-crossing.  This means $\CM_1(\beta_0)$ for $F_r$ between the walls $H_l$ and $H_{l+1}$ is the same as that for a fiber below the wall $H_0$, which is the same as a toric fiber in $\cpx \times (\cpx^\times)^{n-1}$. it follows that $n_{\beta_0} = 1$.  Since $\CM_1(\beta)\cong\CM_1(\beta_0)$, we also have $n_{\beta} = 1$.

For $\beta = \beta_\infty + \sum_{i=0}^p \sum_{j=1}^{k_i} n^i_j \beta^i_j$, we consider the winding number of $t$ around $\infty$ instead of around $K$, and the argument is exactly the same.
\end{proof}

Now we are ready to compute the SYZ mirror of $\pi: \bX_t \to B$.

\begin{thm} \label{SYZ thm}
By the construction defined in the beginning of this section, the SYZ mirror of the Lagrangian fibration $\pi: \bX_t \to B$ is
$$\check{\bX} = \{((\bu,\bv),z_1,\ldots,z_{n-1}) \in \cpx^2 \times (\cpx^*)^{n-1}: \bu \bv = g(z)\}$$
where
\begin{align*}
g(z) &= \prod_{i=0}^p \left(1 + \sum_{l = 1,\ldots,k_i} z^{u^i_l} \right) \\
&= \sum_{v \in P \cap \unu^\perp} n_v z^v
\end{align*}
for some explicit positive integers $n_v$ attached to each $v \in P \cap \unu^\perp$, and $n_v = 1$ for all vertices $v$ of $P$.  Notice that this is independent of the deformation parameter $t$.
\end{thm}

\begin{proof}
We have already identified the wall by Proposition \ref{prop:wall}, and it remains to compute the generating functions $\bu$ and $\bv$ defined by Equation \eqref{u} and \eqref{v} respectively.  By the explicit expression of open Gromov-Witten invariants from Theorem \ref{openGW_Fr},
$$ \bu = \left(z_0, z_0\left(1 + \sum_{l = 1,\ldots,k_i} z^{u^0_l} \right), \ldots, z_0\prod_{i=0}^p \left(1 + \sum_{l = 1,\ldots,k_i} z^{u^i_l} \right)\right) $$
and
$$ \bv = \left(z_0^{-1} \prod_{i=0}^p \left(1 + \sum_{l = 1,\ldots,k_i} z^{u^i_l} \right), \ldots, z_0^{-1} \right)$$
on $\check{\pi}^{-1}(B_0 - H)$.  Thus we have the relation
$$ \bu \bv = \prod_{i=0}^p \left(1 + \sum_{l = 1,\ldots,k_i} z^{u^i_l} \right) $$
and hence $\check{X} = \mathrm{Spec}(R)$ is the subvariety defined by the above equation.
\end{proof}

\begin{remark}
The above calculation of open Gromov-Witten invariants matches with the expectation from Gross-Siebert program \cite{gross07} (which reconstructs the mirror manifold from tropical geometry instead of directly using symplectic geometry): the wall-crossing function attached to each wall $H_i$ is
$$ 1 + \sum_{l = 1,\ldots,k_i} z^{u^i_l} $$
and the coefficient of $z^{v}$ for $v \in \unu^\perp$ in $\prod_{i=0}^p \left( 1 + \sum_{l = 1,\ldots,k_i} z^{u^i_l} \right)$ equals to $\sum_{\beta'} n_{\beta_0 + \beta'}$ where the summation is over all holomorphic disc classes $\beta'$ of Maslov index zero with $\partial \beta' = v$.
\end{remark}

\begin{remark}
All the complex geometric information of $\check{\bX}$ is recorded by the hypersurface $g(z) = 0$ in $(\cpx^*)^{n-1}$.  Notice that this is a normal-crossing variety generically.  This matches with the expectation from naive T-duality that big torus is dual to small torus and vice versa.  Due to the assumption that each $R_i$'s in the Minkowski decomposition is a standard simplex, each irreducible component is of the form
$$\left\{1 + \sum_{l = 1,\ldots,k} z^{u_l} = 0\right\} \subset (\cpx^\times)^{n-1},$$
which is the product of a $(k-1)$-dimensional pair-of-pants with $(\cpx^*)^{n-k}$.
\end{remark}

\begin{remark}
We may also consider the dual picture of Minkowski decomposition, namely, decomposition of the dual fan of the polytope.  By considering the pair-of-pant factor of each component of the hypersurface $g(z) = 0$ and taking the tropical limit, we obtain a fan (which is standard simplicial of possibly lower dimension).  The union of all these fans recovers the dual fan of the original polytope.  See Figure \ref{fan-decomp-dP6} in Example \ref{fan-decomp-dP6}.
\end{remark}

\begin{remark}
In the above expression of mirror, only the integer coefficients of $g$ depend on the Minkowski decomposition that we start with.  Geometrically it means different smoothings of $\bX$ correspond to local patches of different limit points of the same stringy K\"ahler moduli.
\end{remark}

As a consequence, the open Gromov-Witten invariants of $F_r$ for $r$ above all the walls $H_i$ are non-zero only for $\beta_v$ where $v$ is a lattice point in $P$.  This is a non-trivial consequence because a priori all stable disc classes of Maslov index two
$$ \beta_0 + \sum_{i=0}^p \sum_{j=1}^{k_i} n^i_j \beta^i_j $$
may have non-trivial open Gromov-Witten invariants.

Now consider the Lagrangian fibration $\underline{\pi}$ (defined in Section \ref{Gfib}) and its disc potential.  There is only one wall and open Gromov-Witten invariants are well-defined away from this wall.  We follow the terminologies of \cite{auroux07} and make the following definition:

\begin{defn}[Product and Clifford tori] \label{product torus}
A regular fiber $T$ of the Lagrangian fibration $\underline{\pi}$ is called to be a product torus if its based point $r = \underline{\pi}(T)$ is above the wall.  It is called to be a Clifford torus if its based point is below the wall.
\end{defn}

There exists a Lagrangian isotopy between the fiber $F_{r}$ of $\pi^K$ at $r = (b,a) \in B$ for $a >> 0$ and a product torus fiber $T$ of $\underline{\pi}$, such that each member in the isotopy never bounds holomorphic discs of Maslov index zero.  Thus their open Gromov-Witten invariants are equal:
$$ n_\beta^{F_r} = n_\beta^{T} $$
where $\beta \in \pi_2(X,T)$ is identified as a disc class in $\pi_2(X,F_r)$ by this isotopy.  As a consequence,

\begin{cor}
Let $T$ be a product torus fiber of $\underline{\pi}$ and $\beta \in \pi_2(\bX_c,T)$.
Write
$$ \beta = \beta_0 + \sum_{i=0}^p \sum_{j=1}^{k_i} n^i_j \beta^i_j. $$
Then $n_\beta^{T} = 1$ when $n^i_j$ is either zero or one and for every $i$ there is at most one $j$ such that $n^i_j = 1$.  Otherwise $n_\beta^{T} = 0$.
\end{cor}

We also consider the disc potential $W^T$ for a product torus fiber $T$.  In Section \ref{sect:ct} we will see that $W^T$ can be obtained from the toric disc potential of the conifold transition of $\bX_t$.

\begin{defn}[Disc potential] \label{defn:W^T}
Let $T$ be a product torus fiber of $\underline{\pi}$.  The disc potential is defined as
$$ W^T = \sum_{\beta \in \pi_2(\bX_c,T)} n_\beta \exp \left(-\int_\beta \omega\right) \mathrm{Hol}_\conn (\partial \beta). $$
\end{defn}

From the above corollary, we have

\begin{cor} \label{W^T}
Let $T$ be a product torus fiber of $\underline{\pi}$ and $\beta \in \pi_2(\bX_c,T)$.  Then the disc potential for $T$ is
$$ W^T = z_0 \sum_{v \in P \cap N} n_v z^v $$
where $n_v$ are the same integers as in Theorem \ref{SYZ thm}.
\end{cor}

\begin{remark}
If we consider the toric variety $\bX_0$, there are $m$ basic disc classes corresponding to vertices of the polytope.  This gives the Laurent polynomial
$$z_0 \sum_{v: \textrm{ vertex of }P} z^v$$
which is NOT equal to $W^T$ in general.
\end{remark}

\section{Local conifold transitions} \label{sect:ct}
Instead of taking Minkowski decompositions of the polytope $P$ to obtain smoothings of toric Calabi-Yau Gorenstein singularities, one can instead consider triangulations of $P$ giving rise to a fan $\Sigma$ which corresponds to a toric resolution $Y = Y_\Sigma$, and construct the mirror of $Y$ via SYZ.  This has been done in \cite{CLL}.  The procedure of degenerating $\bX_t$ to a toric Gorenstein canonical singularity $\bX_0$ and taking toric resolution $Y$ called conifold transition.  On both sides the geometry of Lagrangian fibrations has been explained beautifully in \cite{gross_examples}.

String theorists expect that quantum geometry undergoes a smooth deformation under conifold transition from $Y$ to $\bX_t$, even though there is singularity developed in the procedure of varying from $Y$ to $\bX_t$.  (See Figure \ref{conifold-transition} showing a picture of the K\"ahler moduli.)  Since we have computed all the open Gromov-Witten invariants of the smoothing $\bX_t$, and from \cite{CCLT13} we also have good understanding on open Gromov-Witten invariants of the toric resolution $Y$, we can compare their disc potentials (which capture the quantum geometry relative to a torus fiber) and also their SYZ mirrors.  The phenomenon is the same as that described by Ruan's crepant resolution conjecture \cite{Ruan06} and also its counterpart in the open sector \cite{CCLT12}.  The main difference is that now we consider crepant resolution of an isolated Gorenstein toric singularity which is \emph{not of orbifold type}, and its smoothing is \emph{no longer a toric manifold}.

\begin{figure}[htp]
\begin{center}
\includegraphics[scale=0.5]{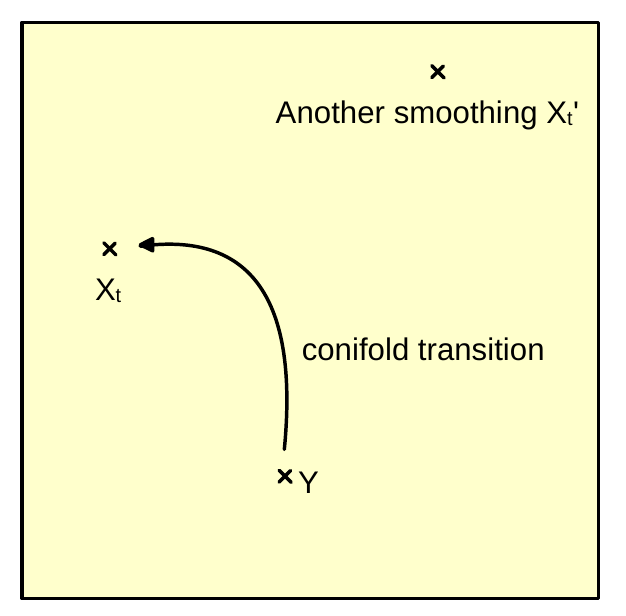}
\end{center}
\caption{The K\"ahler moduli.  For a toric Gorenstein singularity, its resolution lies in a neighborhood of a large volume limit.  Its smoothing gives a conifold limit point in the moduli, and different smoothings gives different limit points.}
\label{conifold-transition}
\end{figure}

The following theorem gives a positive response to the string theorists' expectation.


\begin{thm}[Open conifold transition Theorem] \label{thm:ct}
Let $Y$ be a toric resolution of an isolated Gorenstein canonical singularity given by a lattice polytope $P$, and suppose there exists a Minkowski decomposition of $P$ into standard simplices (of possibly smaller dimension than $P$), such that the corresponding smoothing $\bX_t$ is smooth.  Let $W^Y: (\cpx^\times)^n \to \cpx$ be the disc potential of a moment-map fiber of $Y$, and $W^{\bX_t}: (\cpx^\times)^n \to \cpx$ be the disc potential of a regular fiber of $\underline{\pi}^{\bX_t}$.  Let $q_i$'s be the K\"ahler parameters of $Y$ for $i = 1, \ldots, \tilde{m}-n$, where $\tilde{m}$ is the number of lattice points contained in $P$.  Then there exists an invertible change of coordinates $q(\check{q})$ and specialization of parameters $\check{q} = \underline{\check{q}} \in \cpx^{\tilde{m}-n}$ such that

\begin{enumerate}
\item $\check{W}^Y(\check{q}) := W^Y(q(\check{q}))$ can be analytic continued to all $\check{q} \in \cpx^{\tilde{m}-n}$;
\item
$\check{W}^Y(\check{q} = \underline{\check{q}})$ equals to $W^{\bX_t}$ up to affine change of coordinates on the domain $(\cpx^\times)^n$.
\end{enumerate}

The SYZ mirrors $\check{\bX}$ and $\check{Y}_q$ also have such a relation, that is, the same change of coordinates $q(\check{q})$ and same specialization of variables $\check{q} = \underline{\check{q}}$ gives
$$\check{\bX} = \check{Y}_{q(\check{q})}|_{\check{q} = \underline{\check{q}}}.$$
\end{thm}

\begin{proof}
The disc potential $W^Y$ is given by
$$ W^Y(q) = z_0 \sum_{v \in P \cap N} (1 + \delta_v) z^v $$
where
$$ \delta_v = \sum_{\alpha \not= 0} n_{\beta_v} q^\alpha. $$
By the open mirror theorem \cite[Theorem 1.5]{CCLT13} (Conjecture 5.1 of \cite{CLL}), the disc potential $W^Y$ equals to the Hori-Vafa mirror under the mirror map $q(\check{q})$, which is of the form
$$ W^Y(q(\check{q})) = z_0 \left(z^{v_1} + \ldots + z^{v_{n}} + \sum_{v \in (P \cap N) - \{v_i\}_{i=1}^{n}} \check{q}_v z^v \right) $$
where $\{v_i\}_{i=1}^{n}$ are vertices of a chosen standard simplex in $\unu^\perp$.  We may choose $v_1$ to be a vertex of $P$ and $v_2,\ldots,v_{n}$ be lattice points contained in $P$.  Note that $v_2 - v_1, \ldots, v_{n}-v_1$ form a basis of $\unu^\perp$.  In particular we can write 
$$ v_1 = a_1 (v_2 - v_1) + \ldots + a_{n-1} (v_n - v_{n-1}) $$
for some integers $a_1, \ldots, a_{n-1}$.  Thus $z^{v_1} = (z^{v_2 - v_1})^{a_1} \ldots (z^{v_n - v_{n-1}})^{a_{n-1}}$.

The disc potential $W^{\bX_t}$ has the expression
$$ W^{\bX_t} = z_0 \sum_{v \in P \cap N} n_v z^v $$
by Corollary \ref{W^T}, where $n_v = 1$ for vertices $v$ of $P$.  (In particular $n_{v_1} = 1$.)  

Under the change of coordinates
$$ z_0 \mapsto \lambda_0 z_0, \,\, z^{v_2-v_1} \mapsto \lambda_1 z^{v_2-v_1}, \,\, \ldots, z^{v_n-v_{n-1}} \mapsto \lambda_{n-1} z^{v_n-v_{n-1}} $$
we have
\begin{align*}
z_0z^{v_1} &\mapsto \left(\lambda_0\lambda_1^{a_1}\ldots\lambda_{n-1}^{a_{n-1}}\right) z_0z^{v_1},\\
z_0z^{v_2} &\mapsto  \left(\lambda_0\lambda_1^{a_1}\ldots\lambda_{n-1}^{a_{n-1}}\right)\lambda_1 z_0z^{v_2},\\ 
&\vdots \\
z_0z^{v_{n-1}} &\mapsto \left(\lambda_0\lambda_1^{a_1}\ldots\lambda_{n-1}^{a_{n-1}}\right)\lambda_{n-1} z_0z^{v_{n-1}}.
\end{align*}
Thus by setting
$$\lambda_1 = n_{v_2}, \lambda_{n-1} = n_{v_n}, \lambda_0 = \frac{1}{n_{v_2}^{a_1}\ldots n_{v_n}^{a_{n-1}}},$$
under the change of coordinates $W^Y(q(\check{q}))$ becomes
$$z_0 \left(n_{v_1} z^{v_1} + n_{v_2} z^{v_2} + \ldots + n_{v_{n}} z^{v_{n}} + \sum_{v \in (P \cap N) - \{v_i\}_{i=1}^{n}} c_v \check{q}_v z^v \right)$$
for some constants $c_v$'s.  Then the specialization of parameters
$$\check{q}_v = \frac{n_v}{c_v}$$
for $v \in  (P \cap N) - \{v_i\}_{i=1}^{n}$ equates the above expression with $W^{\bX_t}$.

The SYZ mirrors have similar expressions as $W^Y$ and $W^{\bX_t}$: by the open mirror theorem \cite{CCLT13}, the SYZ mirror of $Y$ is defined by
$$ uv = \left(z^{v_1} + \ldots + z^{v_{n}} + \sum_{v \in (P \cap N) - \{v_i\}_{i=1}^{n}} \check{q}_v z^v \right)$$
while the SYZ mirror of $\bX_t$ is defined by
$$ uv = \prod_{i=0}^p \left(1 + \sum_{l = 1,\ldots,k_i} z^{u^i_l} \right).$$
Thus the mirror map $q(\check{q})$ and the same specialization of parameters give
$$\check{\bX} = \check{Y}_{q(\check{q})}|_{\check{q} = \underline{\check{q}}}.$$
\end{proof}

\begin{remark}
In general the coefficients of the SYZ mirror $\check{Y}_q$ (or the disc potential $W^Y_q$) is a series in $q$, and it may not be legal to specify a value of the K\"ahler parameter $q$ (because the series may not converge at that value).  Thus a change of variable $q(\check{q})$ is necessary in order to have an analytic continuation of $\check{Y}_q$ from the large complex structure limit to the conifold limit.  
\end{remark}

\section{Examples}

\subsection{A-type surface singularities}
Consider the $A_p$ singularity $\cpx^2 / \Z_{p+1}$, which is described by the fan whose rays are generated by $(k,1)$ for $k = 0, \ldots, p+1$.  Then the polytope is the line segment $P = [0,p+1]$ and the fan is obtained by cone over $P$.  By taking the Minkowski decomposition
$$ P = R_0 + \ldots + R_p $$
where $R_i = [0,1]$ for all $i = 0, \ldots, p$, one obtains a smoothing $\bX_t$ of the $A_p$ singularity.  See Figure \ref{fig:An}.

\begin{figure}[htp]
\caption{Smoothing of $A_p$ singularity $\cpx^2 / \Z_{p+1}$ and the Gross fibration.  One can imagine that as the smoothing parameter $c$ tends to $0$, the walls and boundary of the base of the Gross fibration collide together and one gets back the moment map polytope of $\cpx^2 / \Z_{p+1}$.}
  \centering
  \label{fig:An}
   \begin{subfigure}[b]{0.3\textwidth}
   	\centering
    \includegraphics[width=\textwidth]{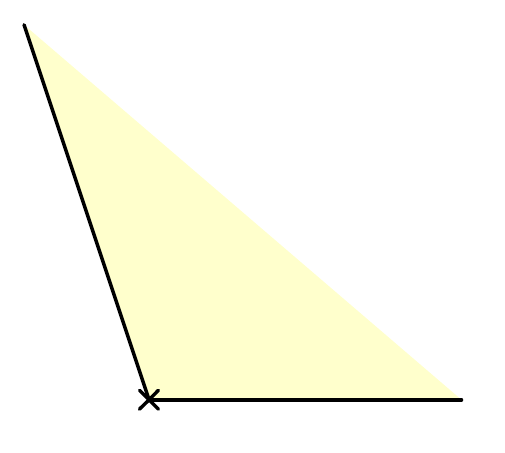}
    \caption{The moment map polytope of $\cpx^2 / \Z_{p+1}$.}
    \label{An-polytope}
   \end{subfigure}
   \hspace{10pt}
   \begin{subfigure}[b]{0.3\textwidth}
   	\centering
    \includegraphics[width=\textwidth]{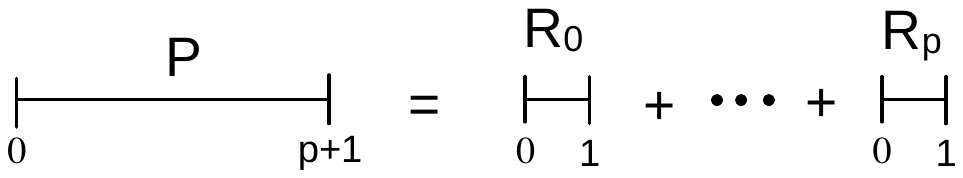}
    \caption{Minkowski decomposition of an interval $[0,p+1]$ into $p+1$ copies of $[0,1]$.}
    \label{Mikowski_decomp_interval}
   \end{subfigure}
   \hspace{10pt}
   \begin{subfigure}[b]{0.3\textwidth}
   	\centering
    \includegraphics[width=\textwidth]{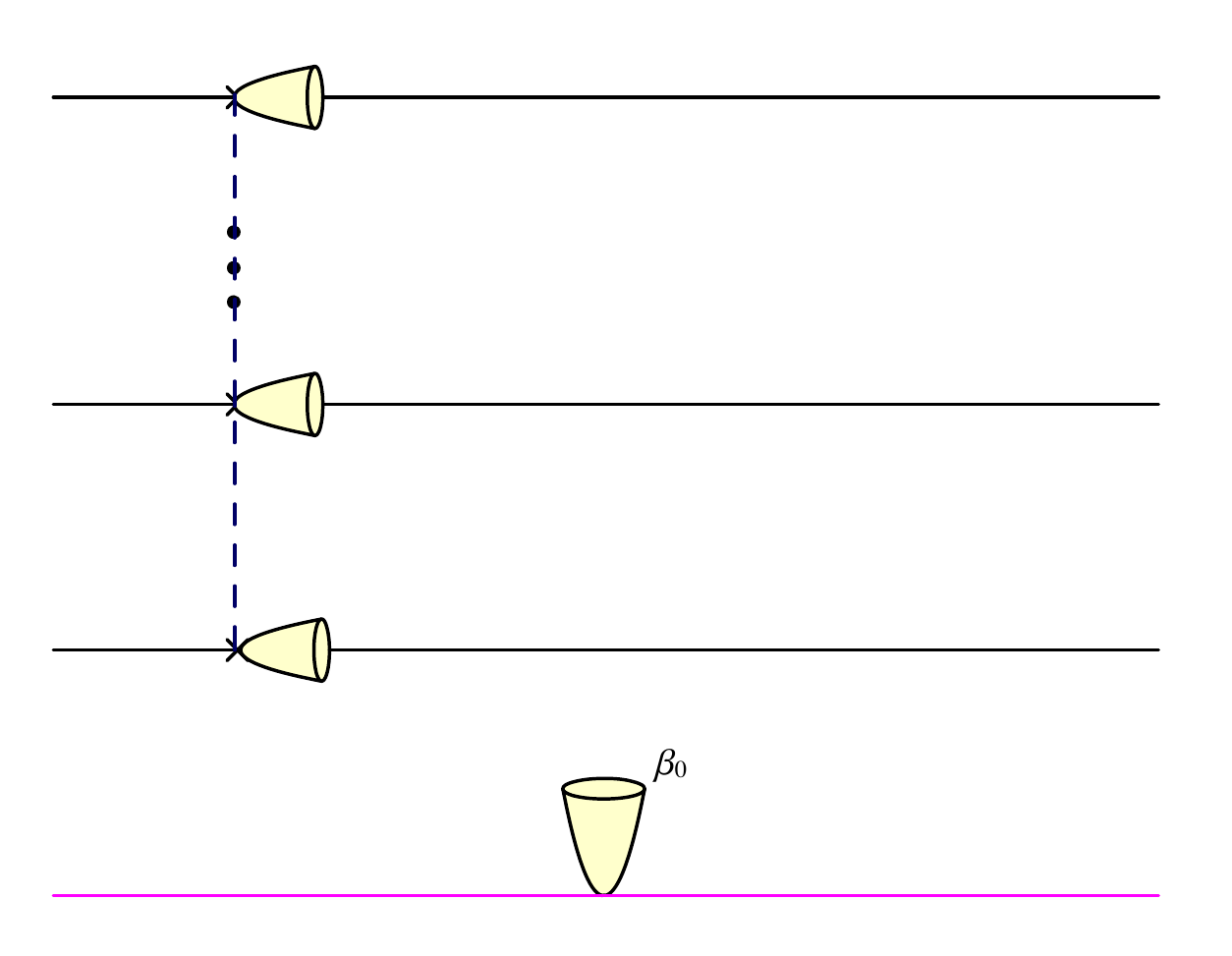}
    \caption{The Gross fibration over $\bX_c$ and the basic holomorphic discs.  There are Lagrangian two-spheres whose images in the base are the dotted lines between the critical points (shown as crosses).}
    \label{GfibAn}
   \end{subfigure}
\end{figure}

The Gross fibration $\pi^K$ has $p+1$ parallel walls $H_i$'s for $i=0,\ldots,p$ in the base, and each of them contains a singular value of the fibration (see Figure \ref{GfibAn}).  There is an $A_p$ chain of Lagrangian two-spheres hitting the singular fibers, and they do not contribute to computation of open Gromov-Witten invariants (in contrast to the other side of resolution).  We can also deform the fibration and make all the walls collapse to one (described in Section \ref{Gfib}), and this is denoted as $\underline{\pi}$.  Then there is one interior singularity left, and the singular fiber is depicted in Figure \ref{AnLagsphere}.

\begin{figure}[htp]
\begin{center}
\includegraphics[scale=0.5]{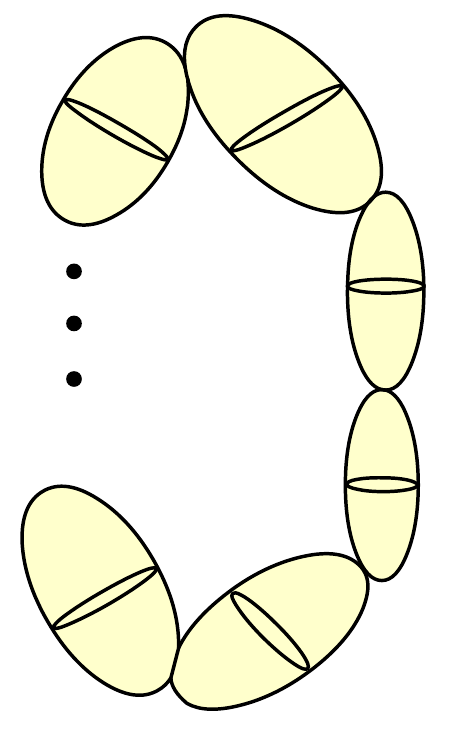}
\end{center}
\caption{The singular fiber of the Lagrangian fibration $\underline{\pi}$ in which all walls collapse to a hyperplane in the base.}
\label{AnLagsphere}
\end{figure}

From Theorem \ref{SYZ thm}, the SYZ mirror is
$$ uv = (1+z)^{p+1} = \sum_{k=0}^{p+1} \mathrm{C}^{\,p+1}_{\,k} z^k $$
for $(u,v,z) \in \C^2 \times \C^\times$, which is again the $A_p$ singularity.  Thus we see that $A_p$ singularity is self-mirror in this sense.  Moreover we get a correspondence between the Minkowski decomposition of an interval and the polynomial factorization $(1+z)^{p+1} = \sum_{k=0}^{p+1} \mathrm{C}^{\,p+1}_{\,k} z^k$.

On the other side of the transition, we consider toric resolution of $A_n$ singularity.  SYZ and relevant open Gromov-Witten invariants have been computed in \cite{LLW_surfaces}.  The moment map polytope is shown in Figure \ref{An resol}.  The SYZ mirror is
$$ uv = (1+z)(1+q_1z)(1+q_1q_2z)\ldots(1+q_1\ldots q_{p}z). $$

\begin{figure}[htp]
\caption{The moment map and Gross fibration of resolution of $A_p$ singularity.  Each line segment on the boundary corresponds to a holomorphic sphere, whose K\"ahler parameter is denoted as $q_i$.}
  \centering
  \label{An resol}
   \begin{subfigure}[b]{0.3\textwidth}
   	\centering
    \includegraphics[width=\textwidth]{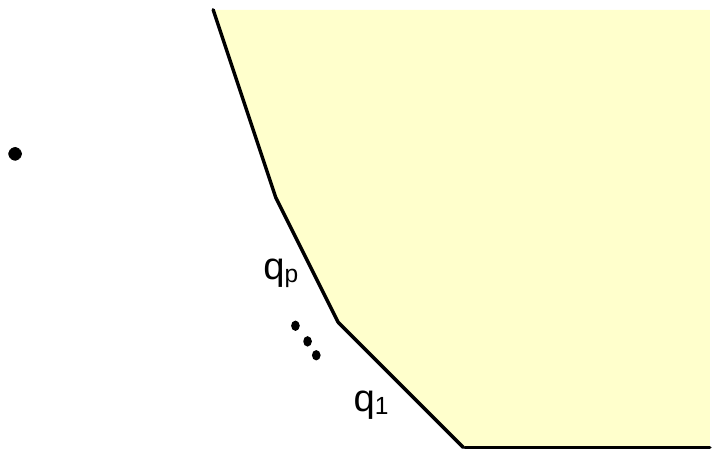}
    \caption{The moment map polytope of resolution of $A_n$ singularity.}
    \label{An-resol-polytope}
   \end{subfigure}
   \hspace{10pt}
   \begin{subfigure}[b]{0.3\textwidth}
   	\centering
    \includegraphics[width=\textwidth]{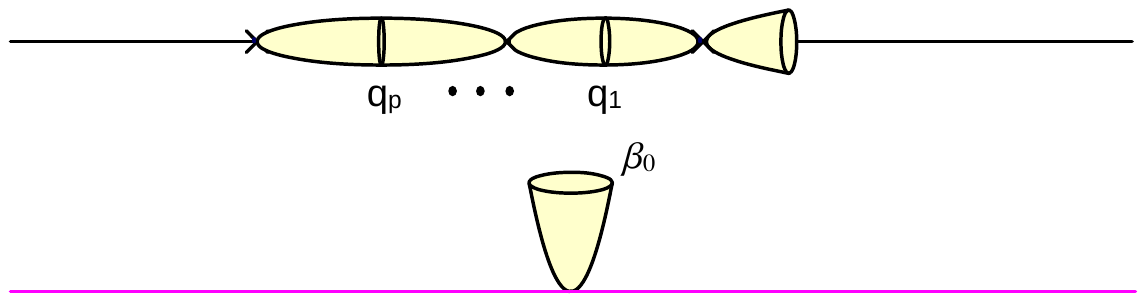}
    \caption{Gross fibration on the resolution.  On this side there is only one wall, and there are holomorphic spheres giving rise to bubbling of open Gromov-Witten invariants.}
    \label{GfibAn-resol}
   \end{subfigure}
\end{figure}

It is easy to see that the specialization of K\"ahler parameters in Theorem \ref{thm:ct} in this case is
$$q_1 = \ldots = q_p = 1$$
and the SYZ mirror of $Y$ reduces to the SYZ mirror of $\bX_t$.
(In this case we do not even need analytic continuation.)  This gives the conifold limit point which is also an orbifold limit in this case.  We can study SYZ via orbidisc invariants (defined in \cite{CP}) of the orbifold $\bX$ instead of its smoothing.  This was done in \cite{CCLT12} and one obtains a different flat coordinates around the conifold limit.

\subsection{A three-dimensional conifold}

Consider the conifold singularity $X$ defined by $xy = zw$ for $x,y,z,w \in \cpx$, which is described by the three-dimensional fan whose rays are generated by
$$(0,0,1), (1,0,1), (0,1,1), (1,1,1).$$
Then the corresponding polytope is the square $P = [0,1] \times [0,1]$ and the fan is obtained by cone over $P$.  By taking the Minkowski decomposition
$$ P = R_0 + R_1 $$
where $R_0 = [1,0]$ and $R_1 = [0,1]$, one obtains a smoothing $\bX_t$ of the conifold singularity.  See Figure \ref{fig:conifold}.  $bX_t$ can be identified with $T^* \bS^3$.

\begin{figure}[htp]
\caption{The three-dimensional conifold, its smoothing and the Gross fibration.  A Lagrangian three-sphere comes up in the smoothing, and its image under the Gross fibration is the dotted line as shown in the diagram.}
  \centering
  \label{fig:conifold}
   \begin{subfigure}[b]{0.3\textwidth}
   	\centering
    \includegraphics[width=\textwidth]{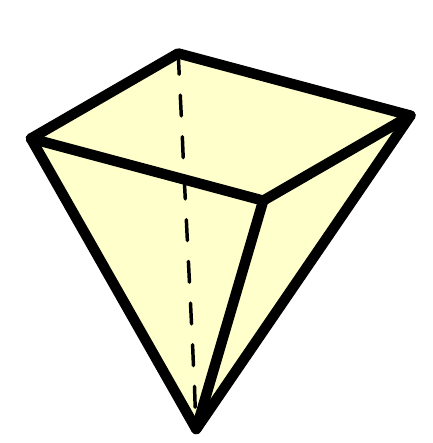}
    \caption{The moment map polytope.}
    \label{conifold-polytope}
   \end{subfigure}
   \hspace{10pt}
   \begin{subfigure}[b]{0.3\textwidth}
   	\centering
    \includegraphics[width=\textwidth]{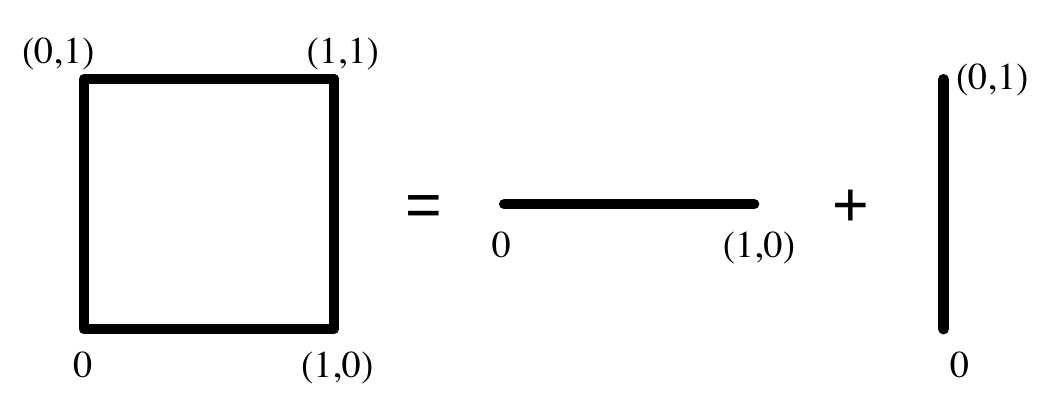}
    \caption{Minkowski decomposition.}
    \label{conifold-decomp}
   \end{subfigure}
   \hspace{10pt}
   \begin{subfigure}[b]{0.3\textwidth}
   	\centering
    \includegraphics[width=\textwidth]{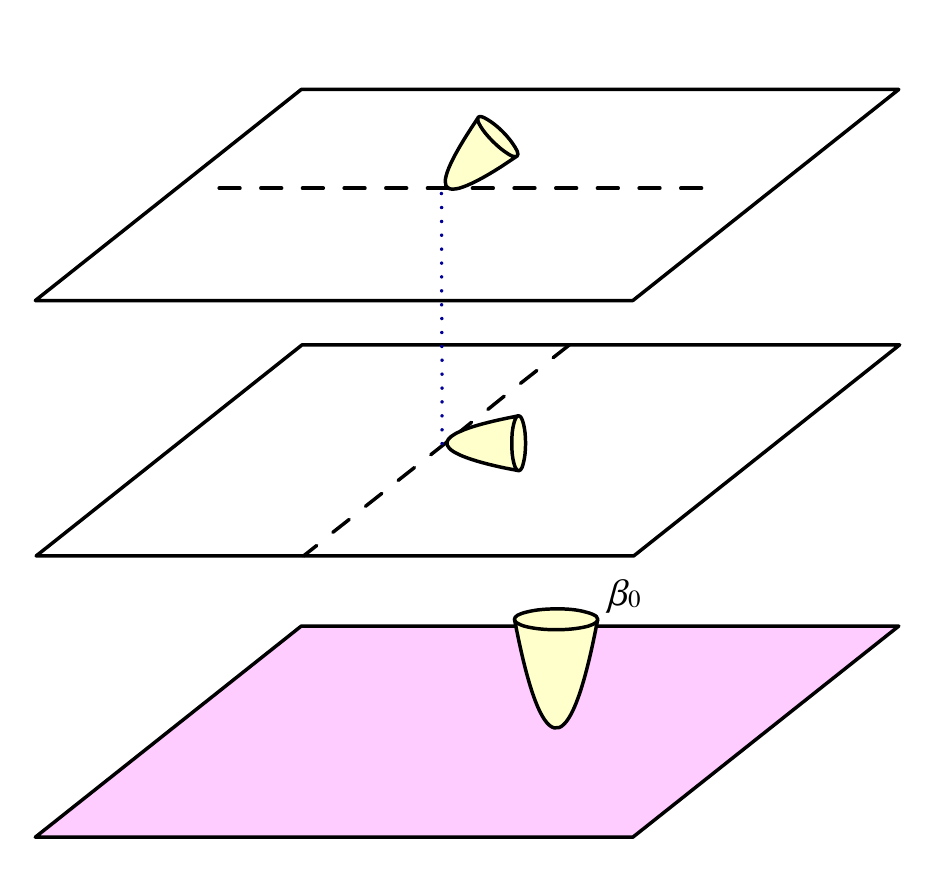}
    \caption{Gross fibration and the basic holomorphic discs.}
    \label{Gfibconifold}
   \end{subfigure}
\end{figure}

The Gross fibration $\pi^K$ has two parallel walls $H_0$ and $H_1$ in the base, and the discriminant loci are the boundary of the base, and two lines contained in these two planes respectively (see Figure \ref{GfibAn}).  There is a Lagrangian three sphere whose image under $\pi^K$ is a vertical line segment joining the two planes.  Under the identification $\bX_t \cong T^*\bS^3$, the Lagrangian sphere is the zero section of $T^*\bS^3 \to \bS^3$.  There are no holomorphic two-sphere.

The Lagrangian fibration $\underline{\pi}$ is formed by collapsing the two walls into one.  See Figure \ref{deg_3torus} for the topological type of the singular Lagrangian fibers over the discriminant locus (which is a cross consisting of two lines) in the wall.

\begin{figure}[htp]
\caption{Singular fibers of $\underline{\pi}$ which are degenerate tori.}
  \centering
  \label{deg_3torus}
   \begin{subfigure}[b]{0.3\textwidth}
   	\centering
    \includegraphics[scale=0.5]{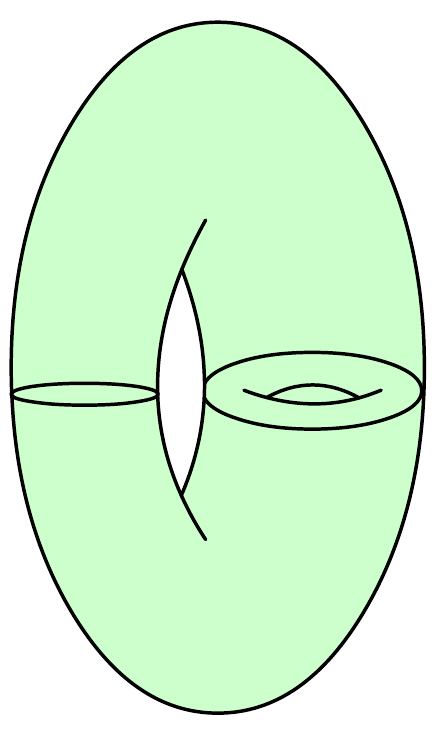}
    \caption{A sigular fiber over a generic point of the cross, the discriminant locus.  It is a two-torus fibration over the circle with one singular fiber.}
    \label{deg_3torus1}
   \end{subfigure}
   \hspace{10pt}
   \begin{subfigure}[b]{0.3\textwidth}
   	\centering
    \includegraphics[scale=0.4]{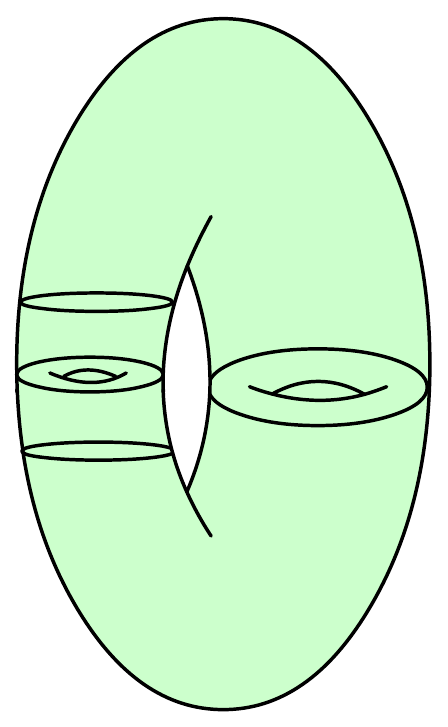}
    \caption{A sigular fiber over the central point of the cross.  It is a two-torus fibration over the circle with two singular fibers.}
    \label{deg_3torus2}
   \end{subfigure}
\end{figure}

From Theorem \ref{SYZ thm}, the SYZ mirror is
$$ uv = (1+z_1)(1+z_2) = 1 + z_1 + z_2 + z_1z_2.$$
The Minkowski decomposition shown in Figure \ref{conifold-decomp} corresponds to the polynomial factorization $(1+z_1)(1+z_2) = 1 + z_1 + z_2 + z_1z_2$.

On the other side of the transition we may consider $\CO_{\bP^1}(-1) \oplus \CO_{\bP^1}(-1)$ which is a toric resolution of $X$ (see Figure \ref{O(-1)+O(-1)}).  SYZ and relevant open Gromov-Witten invariants have been computed in Section 5.3.2 of \cite{CLL}.  The SYZ mirror is
$$ uv = 1+z_1+z_2+ q z_1z_2 $$
where $q$ is the K\"ahler parameter corresponding to the size of the zero section $\bP^1 \subset \CO_{\bP^1}(-1) \oplus \CO_{\bP^1}(-1)$.

\begin{figure}[htp]
\begin{center}
\includegraphics[scale=0.5]{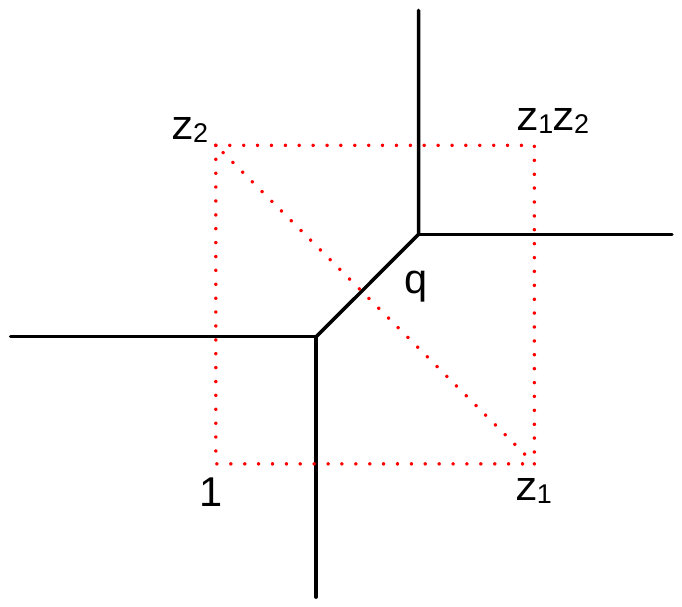}
\end{center}
\caption{$\CO_{\bP^1}(-1) \oplus \CO_{\bP^1}(-1)$.  The dotted line shows the fan picture (projected to the plane) and the solid line shows the polytope picture.  The line segment corresponds to the zero section $\bP^1$ marked by the K\"ahler parameter $q$.}
\label{O(-1)+O(-1)}
\end{figure}

Thus the specialization of K\"ahler parameters in Theorem \ref{thm:ct} in this case is
$$q = 1.$$

\subsection{Cone over del Pezzo surface of degree six}
Let $N = \Z^3$, $\unu = (0,0,1) \in M$.  The polytope $P \in \unu^\perp_\R$ has corners $v_1 = (0,0), v_2 = (1,0), v_3 = (2,1), v_4 = (2,2), v_5 = (1,2)$ and $v_6 = (0,1)$, and $\sigma \subset \R^3$ is the cone over $P$.  This is Example 3.1 of \cite{gross_examples}.  There are two different Minkowski decompositions as shown in Figure \ref{decomp_dP6} giving rise to two different ways of smoothings of $X = X_\sigma$.

The Gross fibrations and holomorphic discs are shown in Figure \ref{fig:GfibdP6}.  In the first smoothing, there are Lagrangian three-spheres whose images are vertical line segments between two consecutive walls (shown as a dotted line in Figure \ref{fig:GfibdP6-1}).  In the second smoothing, there is a Lagrangian cone over the two torus whose image is a vertical line segment between the two walls (shown as a dotted line in Figure \ref{fig:GfibdP6-2}).

\begin{figure}[htp]
\begin{center}
\includegraphics[scale=0.8]{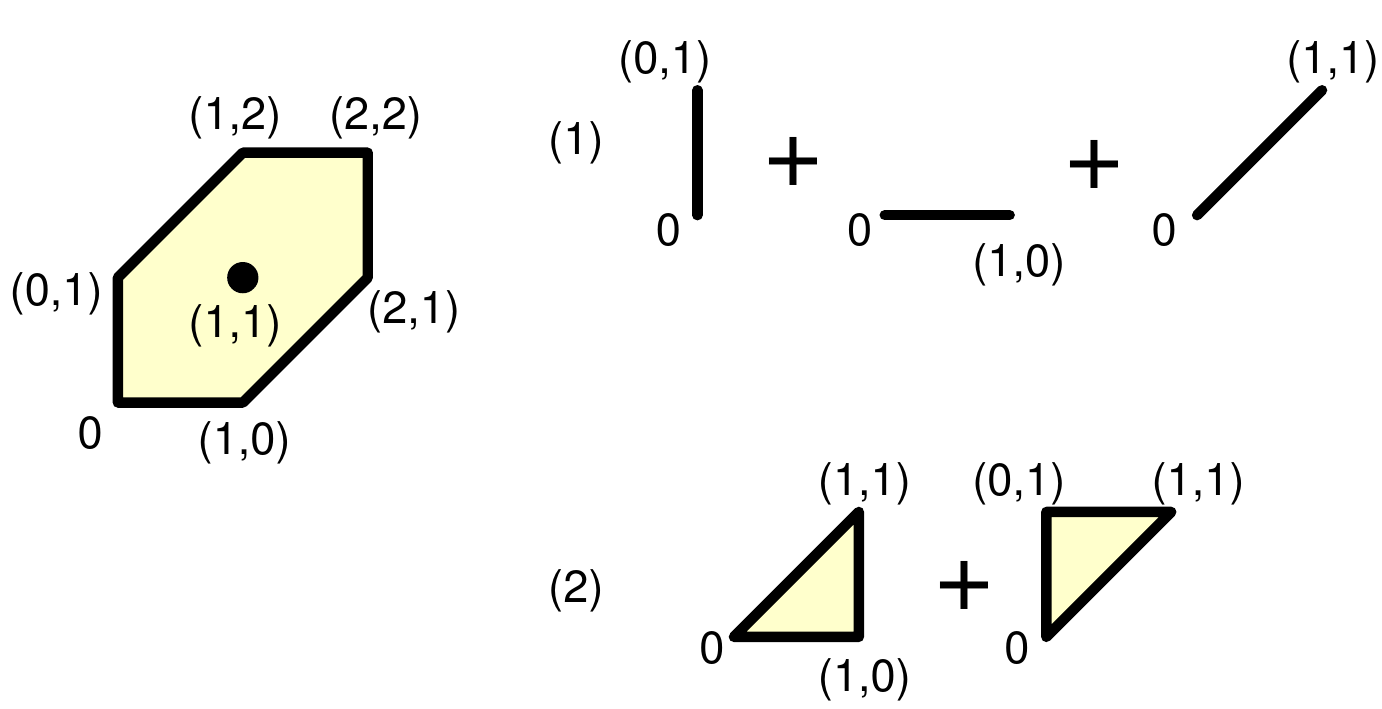}
\end{center}
\caption{Two possible Minkowski decompositions of the fan polytope of $dP_6$.}
\label{decomp_dP6}
\end{figure}

\begin{figure}[htp]
\caption{Cone over $dP_6$ and the Gross fibrations over its two different smoothings.  The dotted lines are the discriminant loci.   The holomorphic disc emanated from the boundary divisor at the bottom is in the class $\beta_0$ which has Maslov index two, and the discs emanated from singular fibers have Maslov index zero.}
  \centering
  \label{fig:GfibdP6}
   \begin{subfigure}[b]{0.3\textwidth}
   	\centering
    \includegraphics[width=\textwidth]{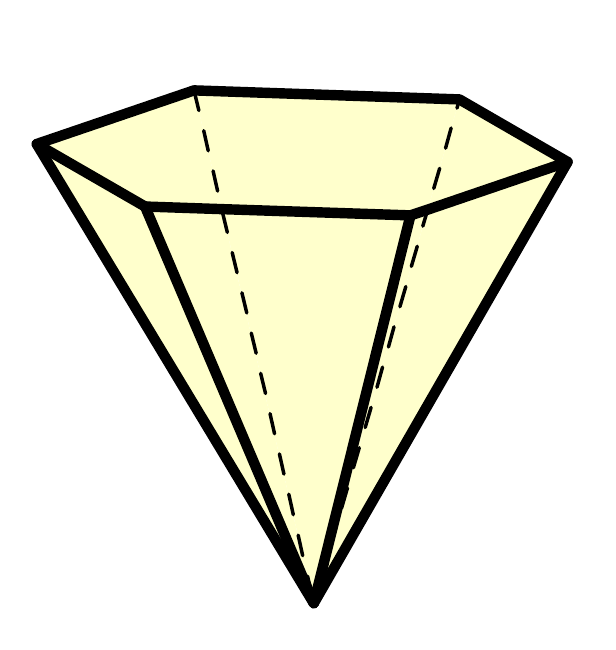}
    \caption{The moment map polytope.  In this example it has the same shape as the fan polytope.}
    \label{fan-polytope-cone_dP6}
   \end{subfigure}
   \hspace{10pt}
   \begin{subfigure}[b]{0.3\textwidth}
   	\centering
    \includegraphics[width=\textwidth]{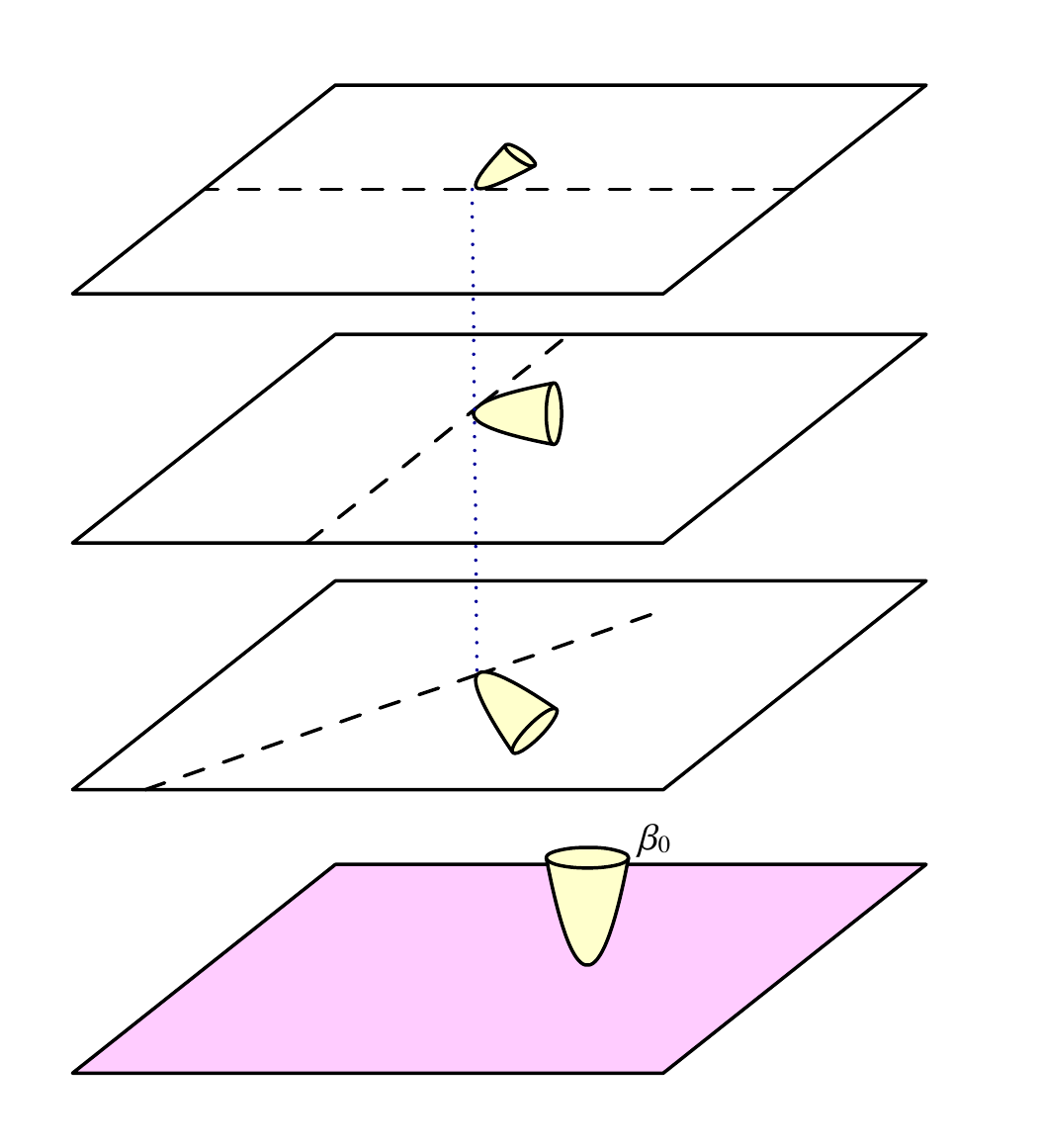}
    \caption{Gross fibration of the first smoothing.}
    \label{fig:GfibdP6-1}
   \end{subfigure}
   \hspace{10pt}
   \begin{subfigure}[b]{0.3\textwidth}
   	\centering
    \includegraphics[width=\textwidth]{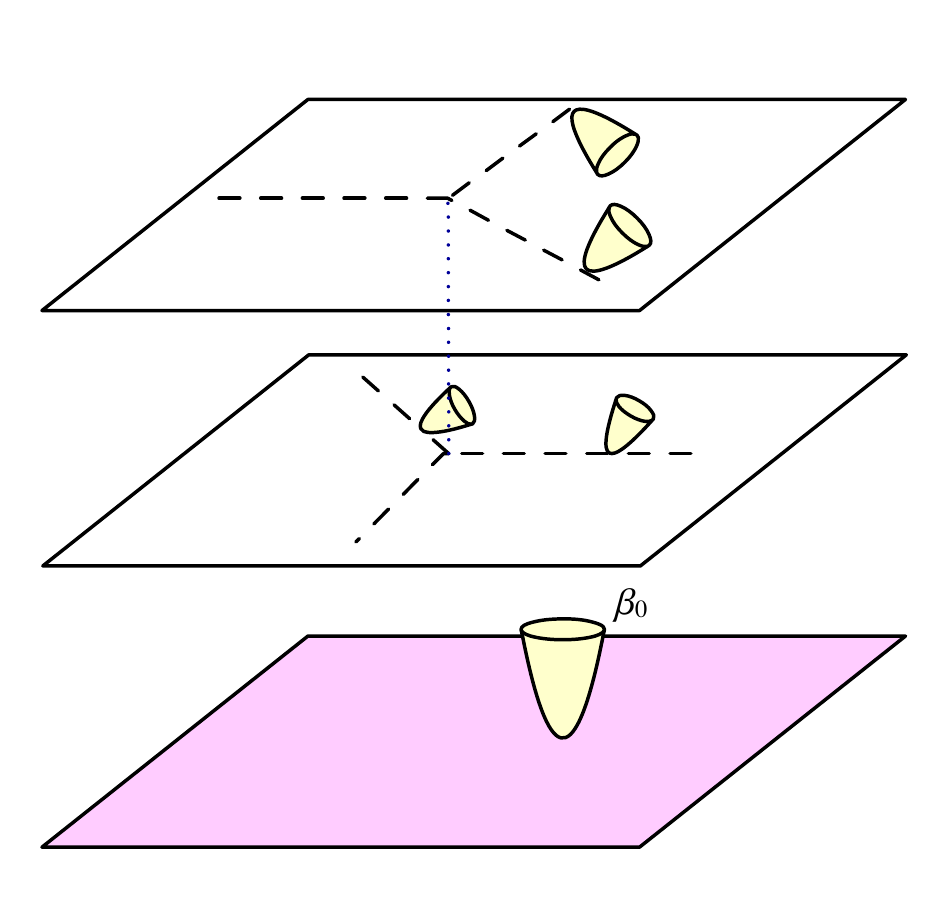}
    \caption{Gross fibration of the second smoothing.}
    \label{fig:GfibdP6-2}
   \end{subfigure}
\end{figure}

From Theorem \ref{SYZ thm}, the SYZ mirrors are
$$ uv = g(z_1,z_2) $$
where the first choice of smoothing gives 
$$ g = g^{\mathrm{con}}_1(z_1,z_2) = (1 + z_1)(1 + z_2) (1 + z_1z_2) = 1 + z_1 + z_2 + 2 z_1 z_2 + z_1^2 z_2 + z_1^2 z_2^2 + z_1 z_2^2 $$
and the second choice of smoothing gives
$$ g = g^{\mathrm{con}}_2(z_1,z_2) = (1 + z_1 + z_1 z_2) (1 + z_2 + z_1 z_2) = 1 + z_1 + z_2 + 3 z_1 z_2 + z_1^2 z_2 + z_1^2 z_2^2 + z_1 z_2^2. $$

These two equations differ only by the coefficients of $z_1 z_2$.  Thus we see that Minkowski decompositions of the polytope in this example correspond to (integral) factorizations of polynomials
for some positive integer $k$.  Since we only allow Minkowski decompositions by standard simplices, the factors that we allow are of the form $1 + \sum_i z^{u_i}$ for certain simplex with vertices $0$ and $u_i$'s. 

For the second choice of smoothing, the tropical diagrams of the components $1 + z_1 + z_1 z_2 = 0$ and $1 + z_2 + z_1 z_2 = 0$ are shown in Figure \ref{fan-decomp-dP6}.  We see that they give a decomposition of the dual fan of the polytope $P$.

\begin{figure}[htp]
\begin{center}
\includegraphics[scale=0.8]{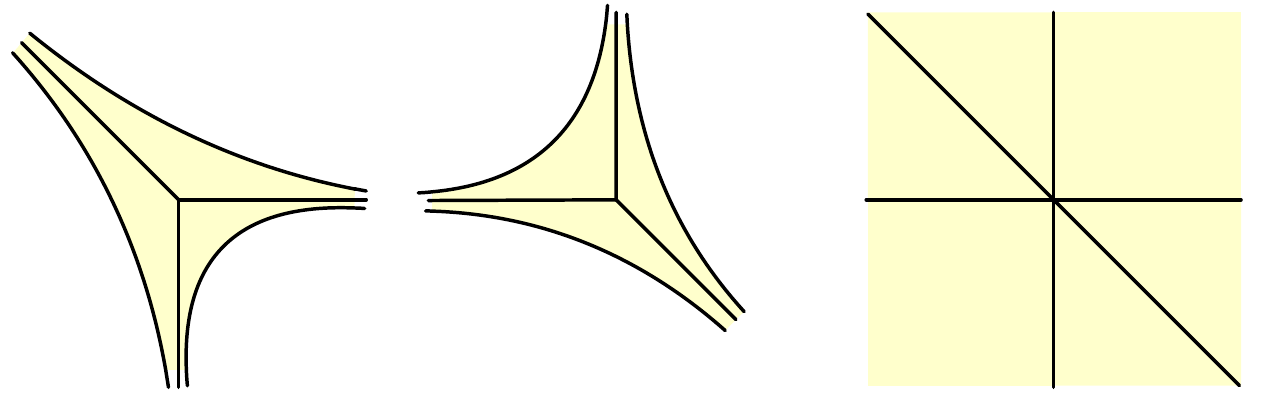}
\end{center}
\caption{The left and middle diagrams show the tropical curves and amoebas associated to $1 + z_1 + z_1 z_2 = 0$ and $1 + z_2 + z_1 z_2 = 0$.  The diagram on the right is the dual fan of the polytope $P$ in Figure \ref{decomp_dP6}.}
\label{fan-decomp-dP6}
\end{figure}

On the other side of the transition, we consider toric resolution of the singularity $\bX$.  There are three choices $Y_A, Y_B, Y_C$ which differ from each other by flops as shown in Figure \ref{fig:dP6-resol}.  In the (complexified) K\"ahler moduli, they correspond to different chambers around the large volume limit.  Their SYZ mirrors are of the same form $uv = g(z_1,z_2)$, where
$$ \widetilde{g^{\mathrm{res}}_A}(z_1,z_2) = z_2 + (1+\delta_A(q^A)) z_1 z_2 + z_1 z_2^2 + q^A_6 + q^A_1 q^A_6 z_1 + q^A_1 q^A_2 z_1^2 z_2 + q^A_5 z_1^2 z_2^2 $$
and $q^A_1 q^A_6 = q^A_3 q^A_4$, $q^A_1 q^A_2 = q^A_4 q^A_5$, $q^A_2 q^A_3 = q^A_5 q^A_6$;
$$ \widetilde{g^{\mathrm{res}}_B}(z_1,z_2) = z_2 + (1+\delta_B(q^B)) z_1 z_2 + z_1 z_2^2 + q^B_6 + q^B_1 q^B_6 z_1 + q^B_1 q^B_2 z_1^2 z_2 + q^B_4 q^B_1 q^B_2 z_1^2 z_2^2 $$
and $q^B_1 q^B_6 = q^B_3$, $q^B_1 q^B_2 = q^B_5$, $q^B_2 q^B_3 = q^B_5 q^B_6$;
$$ \widetilde{g^{\mathrm{res}}_C}(z_1,z_2) = z_2 + (1+\delta_C(q^C))z_1 z_2 + z_1 z_2^2 + q^C_1 q^C_6 + q^C_6 z_1 + q^C_2 z_1^2 z_2 + q^C_2 q^C_4 z_1^2 z_2^2 $$
and $q^C_6 = q^C_3$, $q^C_2 = q^C_5$, $q^C_2 q^C_3 = q^C_5 q^C_6$.  The K\"ahler parameters $q_i$'s correspond to the holomorphic spheres as labelled in Figure \ref{fig:dP6-resol}.  The K\"ahler moduli has dimension four.

\begin{figure}[htp]
\caption{The three toric Calabi-Yau resolutions.  The dotted lines show the polytope $P$ together with its triangulations.  Each vertex corresponds to a monomial of the mirror.  The solid lines show the moment map polytope projected to the plane, and each line segment corresponds to a holomorphic sphere which is labelled by a K\"ahler parameter measuring its area.}
  \centering
  \label{fig:dP6-resol}
   \begin{subfigure}[b]{0.3\textwidth}
   	\centering
    \includegraphics[width=\textwidth]{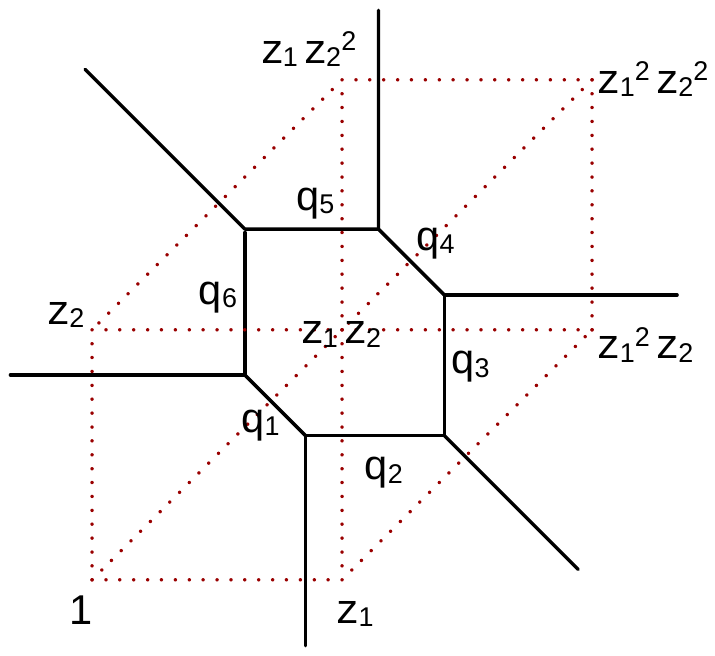}
    \caption{$Y_A$.  This is the total space of canonical line bundle of the del-Pezzo surface of degree six.}
    \label{fig:dP6-resol-A}
   \end{subfigure}
   \hspace{10pt}
   \begin{subfigure}[b]{0.3\textwidth}
   	\centering
    \includegraphics[width=\textwidth]{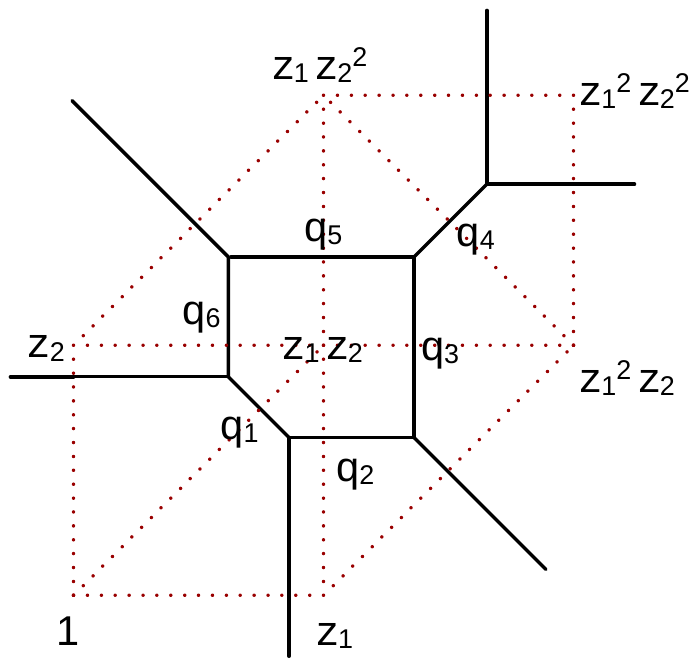}
    \caption{$Y_B$.}
    \label{fig:dP6-resol-B}
   \end{subfigure}
      \hspace{10pt}
   \begin{subfigure}[b]{0.3\textwidth}
   	\centering
    \includegraphics[width=\textwidth]{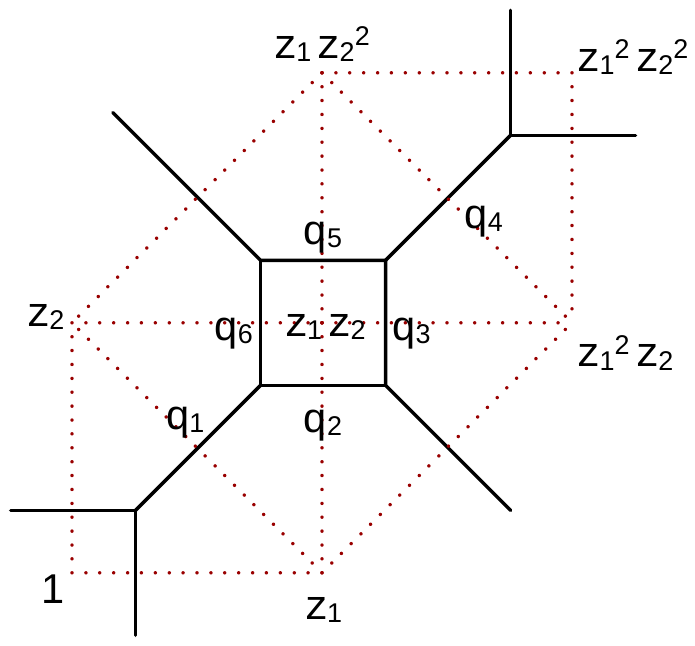}
    \caption{$Y_C$.}
    \label{fig:dP6-resol-C}
   \end{subfigure}
\end{figure}

By the change of coordinates $z_1 \mapsto (1+\delta(q))^{-1} z_1$, $z_2 \mapsto (1+\delta(q)) z_2$ and $u \mapsto (1+\delta(q)) u$, all the above expressions are brought to the form
$$ g_t(z_1,z_2) = z_2 + z_1 z_2 + z_1 z_2^2 + t_1 + t_2 z_1 + t_3 z_1^2 z_2 + t_4 z_1^2 z_2^2$$
and so all of them belong to the same mirror family of complex varieties around the same large complex structure limit.  The SYZ mirrors after such change of coordinates are denoted as $g^{\mathrm{res}}_A, g^{\mathrm{res}}_B, g^{\mathrm{res}}_C$ respectively.

By the open mirror theorem for toric Calabi-Yau manifolds \cite{CCLT13}, $g^{\mathrm{res}}$ can be computed from the mirror maps and expressed as the following:
\begin{align*}
g^{\mathrm{res}}_A(z_1,z_2) &= z_2 + z_1 z_2 + z_1 z_2^2 + \check{q}^A_6 + \check{q}^A_1 \check{q}^A_6 z_1 + \check{q}^A_1 \check{q}^A_2 z_1^2 z_2 + \check{q}^A_5 z_1^2 z_2^2;\\
g^{\mathrm{res}}_B(z_1,z_2) &= z_2 + z_1 z_2 + z_1 z_2^2 + \check{q}^B_6 + \check{q}^B_1 \check{q}^B_6 z_1 + \check{q}^B_1 \check{q}^B_2 z_1^2 z_2 + \check{q}^B_1 \check{q}^B_2 \check{q}^B_4 z_1^2 z_2^2;\\
g^{\mathrm{res}}_C(z_1,z_2) &= z_2 + z_1 z_2 + z_1 z_2^2 + \check{q}^C_1 \check{q}^C_6 + \check{q}^C_6 z_1 + \check{q}^C_2 z_1^2 z_2 + \check{q}^C_2 \check{q}^C_4 z_1^2 z_2^2;.
\end{align*}
where $q^A$ and $\check{q}^A$ (resp. $B,C$) are related by the mirror maps $q^{A}_i = q^{A}_i(\check{q}^{A})$ (resp. $B,C$).  $\check{q}_i$'s satisfy the same relations as $q_i$'s:
\begin{align*}
\check{q}^A_1 \check{q}^A_6 = \check{q}^A_3 \check{q}^A_4&; \check{q}^A_1 \check{q}^A_2 = \check{q}^A_4 \check{q}^A_5; \check{q}^A_2 \check{q}^A_3 = \check{q}^A_5 \check{q}^A_6.\\
\check{q}^B_1 \check{q}^B_6 = \check{q}^B_3&; \check{q}^B_1 \check{q}^B_2 = \check{q}^B_5; \check{q}^B_2 \check{q}^B_3 = \check{q}^B_5 \check{q}^B_6.\\
\check{q}^C_6 = \check{q}^C_3&; \check{q}^C_2 = \check{q}^C_5; \check{q}^C_2 \check{q}^C_3 = \check{q}^C_5 \check{q}^C_6. 
\end{align*}
We can see that $g^{\mathrm{res}}_A$ equals to $g^{\mathrm{res}}_B$ by the change of variables
$$(\check{q}_4^{A})^{-1} = \check{q}_4^{B}; \check{q}_1^{A} = \check{q}_1^{B}; \check{q}_2^{A} = \check{q}_2^{B};\check{q}_6^{A} = \check{q}_6^{B},$$
and $g^{\mathrm{res}}_B$ equals to $g^{\mathrm{res}}_C$ by the change of variables
$$(\check{q}_1^{B})^{-1} = \check{q}_1^{C}; \check{q}_3^{B} = \check{q}_3^{C}; \check{q}_4^{B} = \check{q}_4^{C}; \check{q}_5^{B} = \check{q}_5^{C}.$$
Thus the mirror complex variety undergoes no topological change under the flops, which is a well-known prediction by string theorists. 

Going back to smoothings of $\bX$, by the change of coordinates $z_1 \mapsto z_1/2$, $z_2 \mapsto 2 z_2$ and $u \mapsto 2u$, $uv = g^{\mathrm{con}}_1$ is equivalent to
$$ uv = z_2 + z_1 z_2 + z_1 z_2^2 + \frac{1}{2} + \frac{1}{2} z_1 + \frac{1}{2} z_1^2 z_2 + \frac{1}{2} z_1^2 z_2^2.$$
Similarly $uv = g^{\mathrm{con}}_2$ is equivalent to
$$ uv = z_2 + z_1 z_2 + z_1 z_2^2 + \frac{1}{3} + \frac{1}{3} z_1 + \frac{1}{3} z_1^2 z_2 + \frac{1}{3} z_1^2 z_2^2.$$
Thus we see that the SYZ mirrors of the smoothings correspond to two conifold limit points of the complex moduli:
$$(t_1,t_2,t_3,t_4) = \left(\frac{1}{2},\frac{1}{2},\frac{1}{2},\frac{1}{2}\right)$$
and
$$(t_1,t_2,t_3,t_4) = \left(\frac{1}{3},\frac{1}{3},\frac{1}{3},\frac{1}{3}\right).$$
Then the specialization of variables in Theorem \ref{thm:ct} for the conifold transitions are
\begin{align*}
\check{q}^A_1 = 1, \check{q}^A_2 = \check{q}^A_5 = \check{q}^A_6 = \frac{1}{2}; \\
\check{q}^B_1 = \check{q}^B_4 = 1; \check{q}^B_2 = \check{q}^B_6 = \frac{1}{2};\\
\check{q}^C_1 = \check{q}^C_4 = 1; \check{q}^C_2 = \check{q}^C_6 = \frac{1}{2}
\end{align*}
for the first smoothing, and
\begin{align*}
\check{q}^A_1 = 1, \check{q}^A_2 = \check{q}^A_5 = \check{q}^A_6 = \frac{1}{3}; \\
\check{q}^B_1 = \check{q}^B_4 = 1; \check{q}^B_2 = \check{q}^B_6 = \frac{1}{3};\\
\check{q}^C_1 = \check{q}^C_4 = 1; \check{q}^C_2 = \check{q}^C_6 = \frac{1}{3}
\end{align*}
for the second smoothing.

\bibliographystyle{amsplain}
\bibliography{geometry}

\end{document}